\documentclass[11pt]{article}
\usepackage{latexsym,amsfonts,amssymb,amsmath,amsthm}
\usepackage{graphicx}

\newtheorem{theorem}{Theorem}[section]
\newtheorem{definition}[theorem]{Definition}
\newtheorem{lemma}[theorem]{Lemma}
\newtheorem{corollary}[theorem]{Corollary}
\newtheorem{remark}[theorem]{Remark}

\numberwithin{equation}{section}

\def\p{\partial}

\newcommand\RR{\ensuremath{\mathbb{R}}}

\newcommand\NN{\ensuremath{\mathbb{N}}}
\newcommand\PP{\ensuremath{\mathbb{P}}}

\begin{document}
\begin{center}

\textbf{\Large One-Dimensional Random Attractor and Rotation Number
}

 \vskip 0.2cm

\textbf{\Large of the Stochastic Damped Sine-Gordon
Equation
}

\vskip 1cm

{\large Zhongwei Shen$^{a,}$\footnote{The first two authors are
supported by National Natural Science Foundation of China under
Grant 10771139, and the Innovation Program of Shanghai Municipal
Education Commission under Grant 08ZZ70

$^{2}$The third author is partially supported by NSF grant
DMS-0907752

$^{*}$Corresponding Author: zhoushengfan@yahoo.com}, Shengfan
Zhou$^{a, 1, *}$, Wenxian Shen$^{b, 2}$}

\vskip 0.3cm

$^{a}$\textit{Department of Applied Mathematics, Shanghai Normal
University},

\textit{Shanghai 200234, PR China}

$^{b}$\textit{Department of Mathematics and Statistics, Auburn
University},

\textit{Auburn 36849, USA}

\vskip 1cm

\begin{minipage}[c]{13cm}

\noindent \textbf{Abstract}: This paper is devoted to the study of
the asymptotic dynamics of the stochastic damped sine-Gordon
equation with homogeneous Neumann boundary condition. It is shown
that for any positive damping and diffusion coefficients,  the
equation possesses a random attractor, and when the damping and
diffusion coefficients are sufficiently large, the random attractor
is a one-dimensional random horizontal curve regardless of the
strength of noise. Hence its dynamics is not chaotic. It is also
shown that the equation has a rotation number provided that the
damping and diffusion coefficients are sufficiently large, which
implies that the solutions tend to oscillate with the same frequency
eventually and the so called frequency locking is successful.

\vspace{5pt}


\vspace{5pt}


\end{minipage}
\end{center}

\section{Introduction}

\qquad Let $(\Omega ,\mathcal{F},\PP)$ be a probability space ,
where
\[
\begin{split}
\Omega =\{\omega=(\omega_1,\omega_2,\dots,\omega_m)\in
{C}(\mathbb{R},\mathbb{R}^m):\omega (0)=0\},
\end{split}
\]
the Borel $\sigma $-algebra $\mathcal{F}$ on $\Omega $ is generated
by the compact open topology (see \cite{1}), and $\PP$ is the
corresponding Wiener measure on $\mathcal{F}$. Define $(\theta
_t)_{t\in\mathbb{R}}$ on $\Omega $ via
\[
\begin{split}
\theta _t\omega (\cdot )=\omega (\cdot +t)-\omega (t),\quad
t\in\mathbb{R}.
\end{split}
\]
Thus, $(\Omega,\mathcal{F},\PP,(\theta_t)_{t\in\mathbb{R}})$ is an
ergodic metric dynamical system.

Consider the following stochastic damped sine-Gordon equation with
additive noise:
\begin{equation}\label{main-eq}
du_t+\alpha du+(-K\Delta u+\sin
u)dt=fdt+\sum_{j=1}^{m}h_jdW_{j}\quad
\text{in}\,\,U\times\mathbb{R}^{+}
\end{equation}
complemented with the homogeneous Neumann boundary condition
\begin{equation}\label{main-bc}
\frac{\partial u}{\partial n}=0\quad \text{on}\,\,
\partial U\times\mathbb{R}^{+},
\end{equation}
 where $U\subset\mathbb{R}^n$ is a bounded open
set with a smooth boundary $\partial U$, $u=u(x,t)$ is a real
function of $x\in U$ and $t\geq 0$, $\alpha,\,K>0$ are damping
and diffusion coefficients, respectively, $f\in H^{1}(U)$, $h_j\in
H^2(U)$ with $\frac{\p h_j}{\p n}=0$ on $\p U$, $j=1,\dots,m$, and
$\{W_j\}_{j=1}^{m}$ are independent two-sided real-valued Wiener
processes on $(\Omega,\mathcal{F},\PP)$. We identify
$\omega(t)$ with $(W_1(t),W_2(t),\dots,W_m(t))$, i.e.,
\begin{equation*}\label{ommega-eq}
\omega(t)=(W_1(t),W_2(t),\dots,W_m(t)),\,\, t\in\mathbb{R}.
\end{equation*}

Sine-Gordon equations describe the dynamics of  continuous Josephosn junctions (see \cite{11})
and have been widely studied (see \cite{BiFeLoTr}, \cite{BiLo}, \cite{CaKoRe}, \cite{Dic},
 \cite{9}, \cite{Fan2}, \cite{GhTe}, \cite{KoWi}, \cite{11}, \cite{NaKoKa}, \cite{17}, \cite{18}, \cite{21}, \cite{WaZh1}, \cite{WaZh2}, etc.).
Various interesting dynamical scenarios such as subharmonic
bifurcation and chaotic behavior are observed in  damped and driven
sine-Gordon equations (see \cite{BiFeLoTr}, \cite{BiLo},
\cite{NaKoKa}, etc.). Note that interesting dynamics of a
dissipative system occurs in its global attractor (if it exists). It
is therefore of great importance to study the existence and the
structure/dimension of a global attractor of a damped sine-Gordon
equation.

As it is  known,  under various boundary conditions, a deterministic
damped sine-Gordon equation possesses a finite dimensional global
attractor (see \cite{GhTe, 10, 19, 21, WaZh1, WaZh2}). Moreover,
some upper bounds of the dimension of the attractor were obtained in
\cite{GhTe, 21, WaZh1, WaZh2}. In \cite{18, 19}, the authors proved
that  under Neumann boundary condition, when the damping is
sufficiently large, the dimension of the global attractor is one,
which justifies the folklore that there is no chaotic dynamics in a
strongly damped sine-Gordon equation.

Recently, the existence of attractors of stochastic damped
sine-Gordon equations has been studied by several authors (see
\cite{CaKoRe}, \cite{9}, \cite{Fan2}). For example, for the equation
\eqref{main-eq} with Dirichlet boundary condition considered in
\cite{9}, the author  proved the existence of a finite-dimensional
attractor in the random sense. However, the existing works on
stochastic damped sine-Gordon equations deal with Dirichlet boundary
conditions only. The case of a Neumann boundary condition is of
great physical interest. It is therefore important to investigate
both the existence and structure of attractors of stochastic damped
sine-Gordon equations with Neumann boundary conditions. Observe that
there is no bounded attracting sets in such case in the original
phase space due to the uncontrolled space average of the solutions,
which leads to nontrivial dynamics and also some additional
difficulties. Nevertheless, it is still expected that
\eqref{main-eq}-\eqref{main-bc} possesses an attractor in the
original phase space in proper sense.

The objective of the current paper is to provide a study on the
existence and structure of random attractors (see Definition
\ref{random-attractor-def} for the definition of random attractor)
of stochastic damped sine-Gordon equations with Neumann boundary
conditions, i.e. \eqref{main-eq}-\eqref{main-bc}. We will do so in
terms of the random dynamical system generated by
\eqref{main-eq}-\eqref{main-bc} (see Definition
\ref{random-dynamical-system-def} for the definition of random
dynamical system).

The following are the main results of this paper.

\begin{itemize}
\item[(1)] For any  $\alpha>0$ and $K>0$, \eqref{main-eq}-\eqref{main-bc}
possesses a random attractor (see Theorem
\ref{existence-attractor-thm} and Corollary
\ref{existence-attractor-cor}).

\item[(2)] When $K$ and $\alpha$
are sufficiently large, the random attractor of
\eqref{main-eq}-(\ref{main-bc}) is a one-dimensional random
horizontal curve (and hence is one dimensional) (see Theorem
\ref{one-dimension-thm} and Corollary \ref{one-dimension-cor}).

\item[(3)] When $K$ and $\alpha$
are sufficiently large, the rotation number of \eqref{main-eq}
exists (See Theorem \ref{existence-rotation-number-thm} and
Corollary \ref{existence-rotation-number-cor}).
\end{itemize}

The above results make an important contribution to the
understanding of the nonlinear dynamics of stochastic damped
sine-Gordon equations with Neumann boundary conditions. Property (1)
extends the existence result of random attractor in the Dirichlet
boundary case to the Neumann boundary case and shows that system
\eqref{main-eq}-\eqref{main-bc} is dissipative. By property (2), the
asymptotic dynamics of \eqref{main-eq}-\eqref{main-bc} with
sufficiently large $\alpha$ and $K$ is one dimensional regardless of
the strength of noise and hence is not chaotic. Observe that
$\rho\in\RR$ is called the {\it rotation number} of
\eqref{main-eq}-\eqref{main-bc} (see Definition
\ref{rotation-number-def} for detail) if for any solution $u(t,x)$
of \eqref{main-eq}-\eqref{main-bc} and any $x\in U$, the limit
$\lim_{t\to\infty}\frac{u(t,x)}{t}$ exists almost surely  and
\begin{equation*}
\lim_{t\to\infty}\frac{u(t,x)}{t}=\rho \quad {\rm for}\quad a.e. \quad \omega\in\Omega.
\end{equation*}
Property (3) then shows that all the solutions of
\eqref{main-eq}-\eqref{main-bc} tend to oscillate with the same
frequency eventually almost surely and hence frequency locking is
successful in \eqref{main-eq}-\eqref{main-bc} provided that $\alpha$
and $K$ are sufficiently large.

We remark that the results in the current paper also hold for
stochastic damped sine-Gordon equations with periodic boundary
conditions.

It should be pointed out that the dynamical behavior
of variety of systems of the form \eqref{main-eq} have been
studied in \cite{14,15,16,17} for ordinary differential equations,
\cite{18,19} for partial differential equations and \cite{3,13,20}
for stochastic (random) ordinary differential equations. In above
literatures, two main aspects considered are the structure of the
attractor and the phenomenon of frequency locking. For example, in
\cite{20}, the authors studied a class of nonlinear noisy
oscillators. They proved the existence of a random attractor which
is a family of horizontal curves and the existence of a rotation
number which implies the frequency locking.

The rest of the paper is organized as follows. In section 2, we
present  some basic concepts and properties for general random
dynamical systems. In section 3, we provide some basic settings
about \eqref{main-eq}-\eqref{main-bc} and show that it generates a
random dynamical system in proper function space. We prove in
section 4 the existence of a unique random attractor of the random
dynamical system $\phi$ generated by \eqref{main-eq}-\eqref{main-bc}
for any $\alpha,K>0$.  We show in section 5 that the random
attractor of $\phi$ is a random horizontal curve provided that
$\alpha$ and $K$ are sufficiently large. In section 6, we prove the
existence of a rotation number of \eqref{main-eq}-\eqref{main-bc}
provided that $\alpha$ and $K$ are sufficiently large.
\section{General Random Dynamical Systems}

\qquad In this section, we collect some basic knowledge about
general random dynamical systems (see \cite{1,5} for details). Let
$(X,\|\cdot\|_{X})$ be a separable Hilbert space with Borel
$\sigma$-algebra $\mathcal{B}(X)$ and
$(\Omega,\mathcal{F},\PP,(\theta_t)_{t\in\mathbb{R}})$ be the
ergodic metric dynamical system mentioned in section 1.

\begin{definition}\label{random-dynamical-system-def}
A continuous random dynamical system over
$(\Omega,\mathcal{F},\PP,(\theta_t)_{t\in\mathbb{R}})$ is a
$(\mathcal{B}(\mathbb{R}^+)\times\mathcal{F}\times\mathcal{B}(X),\mathcal{B}(X))$-measurable
mapping
\begin{equation*}
\varphi:\mathbb{R}^{+}\times\Omega\times X\rightarrow
X,\quad (t,\omega,u)\mapsto\varphi(t,\omega,u)
\end{equation*}
such that the following properties hold
\begin{itemize}
\item[(1)] $\varphi(0,\omega,u)=u$ for all $\omega\in\Omega$ and $u\in X$;

\item[(2)]
$\varphi(t+s,\omega,\cdot)=\varphi(t,\theta_s\omega,\varphi(s,\omega,\cdot))$
for all $s,t\geq0$ and $\omega\in\Omega$;

\item[(3)] $\varphi$ is continuous in $t$ and $u$.
\end{itemize}
\end{definition}

For given $u\in X$ and $E,F\subset X$, we define
\begin{equation*}
d(u,E)=\inf_{v\in E}\|u-v\|_X
\end{equation*}
and
\begin{equation*}
d_H(E,F)=\sup_{u\in E}d(u,F).
\end{equation*}
$d_H(E,F)$ is called the {\it Hausdorff semi-distance} from $E$ to $F$.

\begin{definition}\label{random-attractor-def}
\begin{itemize}
\item[(1)] A set-valued mapping $\omega\mapsto
D(\omega):\Omega\rightarrow 2^{X}$ is said to be a {\rm random set}
if the mapping $\omega\mapsto d(u,D(\omega))$ is measurable for any
$u\in X$. If $D(\omega)$ is also closed (compact) for each
$\omega\in\Omega$, the mapping $\omega\mapsto D(\omega)$ is called a
{\rm random closed (compact) set}. A random set $\omega\mapsto
D(\omega)$ is said to be {\rm bounded} if there exist $u_0\in X$ and
a random variable $R(\omega)>0$ such that
\begin{equation*}
D(\omega)\subset\{u\in X:\|u-u_0\|_{X}\leq
R(\omega)\}\quad\text{for all}\quad \omega\in\Omega.
\end{equation*}

\item[(2)] A random set $\omega\mapsto D(\omega)$ is called {\rm
tempered} provided for $\PP$-a.s.$\omega\in\Omega$,
\begin{equation*}
\lim\limits_{t\rightarrow\infty}e^{-\beta t}\sup\{\|b\|_{X}:b\in
D(\theta_{-t}\omega)\}=0\quad\text{for all}\quad\beta>0.
\end{equation*}

\item[(3)] A random set $\omega\mapsto B(\omega)$ is said to be a {\rm
random absorbing set} if for any tempered random set $\omega\mapsto
D(\omega)$, there exists $t_0(\omega)$ such that
\begin{equation*}
\varphi(t,\theta_{-t}\omega,D(\theta_{-t}\omega))\subset
B(\omega)\quad\text{for all}\quad t\geq
t_0(\omega),\,\,\omega\in\Omega.
\end{equation*}

\item[(4)] A random set $\omega\mapsto B_1(\omega)$ is said to be a
{\rm random attracting set} if for any tempered random set
$\omega\mapsto D(\omega)$, we have
\begin{equation*}
\lim_{t\rightarrow\infty}d_{H}(\varphi(t,\theta_{-t}\omega,D(\theta_{-t}\omega),B_1(\omega))=0\quad\text{for
all}\quad\omega\in\Omega.
\end{equation*}

\item[(5)] A random compact set $\omega\mapsto A(\omega)$ is said to be
a random attractor if it is an random attracting set and
$\varphi(t,\omega,A(\omega))=A(\theta_t\omega)$ for all
$\omega\in\Omega$ and $t\geq 0$.
\end{itemize}
\end{definition}

\begin{theorem}\label{existence-attractor-thm-rds}
Let $\varphi$ be a continuous random dynamical system over
$(\Omega,\mathcal{F},\PP,(\theta_t)_{t\in\mathbb{R}})$.   If there
is a  tempered random compact attracting set $\omega\mapsto
B_1(\omega)$ of $\varphi$, then $\omega\mapsto A(\omega)$ is a
random attractor of $\varphi$, where
\begin{equation*}
A(\omega)=\bigcap_{t>0}\overline{\bigcup_{\tau\geq
t}\varphi(\tau,\theta_{-\tau}\omega,B_{1}(\theta_{-\tau}\omega))},\quad\omega\in\Omega.
\end{equation*}
Moreover, $\omega\mapsto A(\omega)$ is the unique random attractor
of $\phi$.
\end{theorem}
\begin{proof}
See \cite[Theorem 1.8.1]{5}.
\end{proof}

\section{Basic Settings}

\qquad In this section, we give some basic settings about
(1.1)-(1.2) and show that it generates a random dynamical system.
Define an unbounded operator
\begin{equation}\label{generator-eq}
A:D(A)\equiv\Big\{u\in
H^{2}(U):\frac{\partial u}{\partial n}\Big|_{\partial
U}=0\Big\}\rightarrow L^{2}(U),\,\,u\mapsto-K\Delta u.
\end{equation}
Clearly, $A$ is nonnegative definite and self-adjoint. Its spectral
set consists of only nonnegative eigenvalues, denoted by
$\lambda_i$, satisfying
\begin{equation}\label{lambdas-eq}
0=\lambda_0<\lambda_1\leq\lambda_2\leq\cdots\leq\lambda_i\leq\cdots,\,\,(\lambda_i\rightarrow+\infty\quad\text{as}
\quad i\rightarrow\infty).
\end{equation}
It is well known that $-A$ generates an analytic semigroup of
bounded linear operators $\{e^{-At}\}_{t\geq0}$ on $L^2(U)$ (and
$H^1(U)$). Let $E=H^{1}(U)\times L^{2}(U)$, endowed with the usual
norm
\begin{equation}\label{2a}
\|Y\|_{H^{1}\times L^{2}}=\big(\|\nabla
u\|^{2}+\|u\|^{2}+\|v\|^{2}\big)^{\frac{1}{2}}\quad\text{for}\quad
Y=(u,v)^{\top},
\end{equation}
where $\|\cdot\|$ denotes the usual norm in $L^{2}(U)$ and $\top$
stands for the transposition.

The existence of solutions to problem
\eqref{main-eq}-\eqref{main-bc} follows from \cite{7}. We next
transform the problem \eqref{main-eq}-\eqref{main-bc} to a
deterministic system with a random parameter, and then show that it
generates a random dynamical system.

Let $(\Omega,\mathcal{F},\PP,(\theta_t)_{t\in\mathbb{R}})$
be the ergodic  metric dynamical system in section 1.
 For $j\in\{1,2,\dots,m\}$,
consider the one-dimensional Ornstein-Uhlenbeck equation
\begin{equation*}
dz_j+z_jdt=dW_j(t).
\end{equation*}
Its unique stationary solution is given by
\begin{equation*}
z_j(\theta_t\omega_j)=\int_{-\infty}^{0}e^{s}(\theta_t\omega_j)(s)ds
=-\int_{-\infty}^{0}e^{s}\omega_j(s+t)ds+\omega_j(t),\quad
t\in\mathbb{R}.
\end{equation*}
Note that the random variable $|z_j(\omega_j)|$ is tempered and
the mapping $t\mapsto z_j(\theta_t\omega_j)$ is $\PP$-a.s.
continuous (see \cite{2,8}). More precisely, there is a
$\theta_t$-invariant $\Omega_0\subset\Omega$ with
$\PP(\Omega_0)=1$ such that $t\mapsto z_j(\theta_t\omega_j)$ is
continuous for $\omega\in\Omega_0$ and $j=1,2,\cdots,m$.
 Putting
$z(\theta_t\omega)=\sum_{j=1}^{m}h_jz_j(\theta_t\omega_j)$, which
solves $dz+zdt=\sum_{j=1}^{m}h_jdW_j$.

Now, let $v=u_t-z(\theta_t\omega)$ and take the functional space $E$
into consideration, we obtain the equivalent system of
\eqref{main-eq}-\eqref{main-bc},
\begin{equation} \label{2b}
\left\{ \begin{aligned}
&\dot{u}=v+z(\theta_t\omega),\\
&\dot{v}=-Au-\alpha v-\sin u+f+(1-\alpha)z(\theta_t\omega).
\end{aligned} \right.
\end{equation}
Let $Y=(u,v)^{\top}$, $C=\begin{pmatrix}0&I\\-A&-\alpha
I\end{pmatrix}$, $F(\theta_t\omega,Y)=(z(\theta_t\omega),-\sin
u+f+(1-\alpha)z(\theta_t\omega))^{\top}$, problem (\ref{2b}) has the
following simple matrix form
\begin{equation}\label{2c}
\dot{Y}=CY+F(\theta_t\omega,Y).
\end{equation}
We will consider \eqref{2b} or \eqref{2c} for $\omega\in\Omega_0$
and write $\Omega_0$ as $\Omega$ from now on.

Clearly, $C$ is an unbounded closed operator on $E$ with domain
$D(C)=D(A)\times H^{1}(U)$. It is not difficult to check that the
spectral set of $C$ consists of only following points \cite{19}
\begin{equation*}
\mu_{i}^{\pm}=\frac{-\alpha\pm\sqrt{\alpha^2-4\lambda_i}}{2},\quad
i=0,1,2,\dots
\end{equation*}
and $C$ generates a $C_0$-semigroup of bounded linear operators
$\{e^{Ct}\}_{t\geq0}$ on $E$. Furthermore, let
$F^{\omega}(t,Y):=F(\theta_t\omega,Y)$, it is easy to see that
$F^{\omega}(\cdot,\cdot):\mathbb{R}^{+}\times E\rightarrow E$ is
continuous in $t$ and globally Lipschitz continuous in $Y$ for each
$\omega\in\Omega$. By the classical theory concerning the existence
and uniqueness of the solutions, we obtain (see \cite{12,21})

\begin{theorem}
Consider \eqref{2c}. For each $\omega\in\Omega$ and each $Y_0\in E$,
there exists a unique function $Y(\cdot,\omega,Y_0)\in
C([0,+\infty);E)$ such that $Y(0,\omega,Y_0)=Y_0$ and
$Y(t,\omega,Y_0)$ satisfies the integral equation
\begin{equation}\label{2d}
Y(t,\omega,Y_0)=e^{Ct}Y_0+\int_{0}^{t}e^{C(t-s)}F(\theta_s\omega,Y(s,\omega,Y_0))ds.
\end{equation}
Furthermore, if $Y_0\in D(C)$, there exists $Y(\cdot,\omega,Y_0)\in
C([0,+\infty);D(C))\cap C^{1}((\mathbb{R}^{+},+\infty);E)$ which
satisfies \eqref{2d} and $Y(t,\omega,Y_0)$ is jointly continuous in
$t$, $Y_0$, and is measurable in $\omega$. Then,
$Y:\mathbb{R}^{+}\times\Omega\times E\rightarrow
E\,(\text{or}\,\,\mathbb{R}^{+}\times\Omega\times D(C)\rightarrow
D(C))$ is a continuous random dynamical system.
\end{theorem}

We now define a mapping $\phi:\mathbb{R}^{+}\times\Omega\times
E\rightarrow E\,(\text{or}\,\,\mathbb{R}^{+}\times\Omega\times
D(C)\rightarrow D(C))$ by
\begin{equation}\label{2e}
\phi(t,\omega,\phi_0)=Y(t,\omega,Y_0(\omega))+(0,z(\theta_t\omega))^{\top},
\end{equation}
where $\phi_0=(u_0,u_1)^{\top}$ and
$Y_0(\omega)=(u_0,u_1-z(\omega))^{\top}$. It is easy to see that
$\phi$ is a continuous random dynamical system associated with the
problem \eqref{main-eq}-\eqref{main-bc} on $E\,(\text{or}\,\,D(C))$.
We next show a useful property of just defined random dynamical
systems.

\begin{lemma}\label{periodicity-lm}
Suppose that $p_0=(2\pi,0)^{\top}$. The random dynamical system $Y$
defined in \eqref{2d} is $p_0$-translation invariant in the sense
that
\begin{equation*}
Y(t,\omega,Y_0+p_0)=Y(t,\omega,Y_0)+p_0,\quad
t\geq0,\,\,\omega\in\Omega,\,\,Y_0\in E.
\end{equation*}
\end{lemma}
\begin{proof}
Since $Cp_0=0$ and $F(t,\omega,Y)$ is $p_0$-periodic in $Y$,
$Y(t,\omega,Y_0)+p_0$ is a solution of (\ref{2c}) with initial data
$Y_0+p_0$. Thus, $Y(t,\omega,Y_0)+p_0=Y(t,\omega,Y_0+p_0)$.
\end{proof}

Note that $\mu_{1}^{+}\rightarrow0$ as $\alpha\rightarrow+\infty$,
which will cause some difficulty. In order to overcome it, we
introduce a new norm which is equivalent to the usual norm
$\|\cdot\|_{H^{1}\times L^{2}}$ on $E$ in \eqref{2a}. Here, we only
collect some results about the new norm (see \cite{19} for details).
Since $C$ has at least two real eigenvalues $0$ and $-\alpha$ with
corresponding eigenvectors $\eta_0=(1,0)^{\top}$ and
$\eta_{-1}=(1,-\alpha)^{\top}$, let $E_1=\text{span}\{\eta_0\}$,
$E_{-1}=\text{span}\{\eta_{-1}\}$ and $E_{11}=E_{1}+E_{-1}$. For any
$u\in L^{2}(U)$, define $\bar{u}=\frac{1}{|U|}\int_{U}u(x)dx$, i.e.,
the spatial average of $u$, let $\dot{L}^{2}(U)=\{u\in
L^{2}(U):\bar{u}=0\}$, $\dot{H}^{1}(U)=H^{1}(U)\cap \dot{L}^{2}(U)$
and $E_{22}=\dot{H}^{1}(U)\times\dot{L}^{2}(U)$. It's easy to see
that $E=E_{11}\oplus E_{22}$ and $E_1$ is invariant under $C$. We
now define two bilinear forms on $E_{11}$ and $E_{22}$ respectively.
For $Y_i=(u_i,v_i)^{\top}\in E_{11}$, $i=1,2$, let
\begin{equation}\label{2f}
\langle Y_1,Y_2\rangle_{E_{11}}=\frac{\alpha^2}{4}\langle
u_1,u_2\rangle+\langle
\frac{\alpha}{2}u_1+v_1,\frac{\alpha}{2}u_2+v_2\rangle,
\end{equation}
where $\langle\cdot,\cdot\rangle$ denotes the inner product on
$L^{2}(U)$, and for $Y_i=(u_i,v_i)^{\top}\in E_{22},\,i=1,2$, let
\begin{equation}\label{2g}
\langle Y_1,Y_2\rangle_{E_{22}}=\langle
A^{\frac{1}{2}}u_1,A^{\frac{1}{2}}
u_2\rangle+(\frac{\alpha^2}{4}-\delta\lambda_1)\langle
u_1,u_2\rangle+\langle
\frac{\alpha}{2}u_1+v_1,\frac{\alpha}{2}u_2+v_2\rangle,
\end{equation}
where $A^{\frac{1}{2}}=\sqrt{K}\nabla$ (see \eqref{generator-eq} for
the definition of $A$) and $\delta\in(0,1]$. By the Poincar\'{e}
inequality
\begin{equation*}
\|A^{\frac{1}{2}}u\|^2\geq\lambda_1\|u\|^2,\quad\forall
u\in\dot{H}^{1}(U),
\end{equation*}
\eqref{2g} is then positive definite. Note that for any $Y\in E$,
$\bar{Y}=\int_{U}Y(x)dx\in E_{11}$ and $Y-\bar{Y}\in E_{22}$, thus
we define
\begin{equation}\label{2h}
\langle Y_1,Y_2\rangle_{E}=\langle
\bar{Y}_1,\bar{Y}_2\rangle_{E_{11}}+\langle
Y_1-\bar{Y}_1,Y_2-\bar{Y}_2\rangle_{E_{22}}\quad\text{for}\quad
Y_1,Y_2\in E.
\end{equation}

\begin{lemma}[\cite{19}]
\begin{itemize}
\item[(1)] \eqref{2f} and \eqref{2g} define inner products on $E_{11}$
and $E_{22}$, respectively.

\item[(2)] \eqref{2h} defines an inner product on $E$, and the
corresponding norm $\|\cdot\|_{E}$ is equivalent to the usual norm
$\|\cdot\|_{H^1\times L^2}$ in \eqref{2a}, where
\begin{equation}
\begin{split}
\|Y\|_{E}&=\Big(\frac{\alpha^2}{4}\|u\|^2+\|\frac{\alpha}{2}u+v\|^2+\|A^{\frac{1}{2}}(u-\bar{u})\|^2-\delta\lambda_1\|u-\bar{u}\|^2\Big)^{\frac{1}{2}}\\
&=\Big(\frac{\alpha^2}{4}\|u\|^2+\|\frac{\alpha}{2}u+v\|^2+\|A^{\frac{1}{2}}u\|^2-\delta\lambda_1\|u-\bar{u}\|^2\Big)^{\frac{1}{2}}
\end{split}\label{2i}
\end{equation}
for $Y=(u,v)^{\top}\in E$.

\item[(3)] In terms of the inner product $\langle\cdot,\cdot\rangle_{E}$,
$E_1$ and $E_{11}$ are orthogonal to $E_{-1}$ and $E_{22}$,
respectively.

\item[(4)] In terms of the norm $\|\cdot\|_{E}$, the Lipschitz constant
$L_{F}$ of $F$ in \eqref{2c} satisfies
\begin{equation}
\label{lipschitz-constant-f}
L_{F}\leq\frac{2}{\alpha}.
\end{equation}
\end{itemize}
\end{lemma}

Now let $E_2=E_{-1}\oplus E_{22}$, then $E_2$ is orthogonal to $E_1$
and $E=E_1\oplus E_2$. Thus, $E_2$ is also invariant under $C$.
Denote by $P$ and $Q\,(=I-P)$ the projections from $E$ into $E_1$
and $E_2$, respectively.

\begin{lemma}
\label{projection-lm}
\begin{itemize}
\item[(1)] For any $Y\in D(C)\cap E_2$, $\langle
CY,Y\rangle_{E}\leq-a\|Y\|_{E}^{2}$, where
\begin{equation}\label{a-eq}
a=\frac{\alpha}{2}-\Big|\frac{\alpha}{2}-\frac{\delta\lambda_1}{\alpha}\Big|.
\end{equation}

\item[(2)] $\|e^{Ct}Q\|\leq e^{-at}$ for $t\geq0$.

\item[(3)] $e^{Ct}PY=PY$ for $Y\in E$, $t\geq0$.
\end{itemize}
\end{lemma}
\begin{proof}
See Lemma 3.3 and Corollary 3.3.1 in \cite{19} for (1) and (2). We
now show (3). For $Y\in D(C)\cap E_1$, since
$\frac{d}{dt}e^{Ct}Y=e^{Ct}CY=0$, we have $e^{Ct}Y=e^{C0}Y=Y$. Then,
by approximation, $e^{Ct}Y=Y$ for $u\in E_1,\,\,t\geq0$, since
$D(A)\cap E_1$ is dense in $E_1$. Thus, $e^{Ct}PY=PY$ for $Y\in E$,
$t\geq0$.
\end{proof}

We will need the following lemma and its corollaries.
\begin{lemma}\label{tempered-lm1}
For any $\epsilon>0$, there is a tempered
random variable $r:\Omega\mapsto\mathbb{R}^{+}$ such that
\begin{equation}\label{2j}
\|z(\theta_{t}\omega)\|\leq e^{\epsilon|t|}r(\omega)\quad\text{for
all}\quad t\in\mathbb{R},\,\,\omega\in\Omega,
\end{equation}
where $r(\omega)$, $\omega\in\Omega$ satisfies
\begin{equation}\label{2k}
e^{-\epsilon|t|}r(\omega)\leq r(\theta_t\omega)\leq
e^{\epsilon|t|}r(\omega),\quad t\in\mathbb{R},\,\,\omega\in\Omega.
\end{equation}
\end{lemma}
\begin{proof}
For $j\in\{1,2,\dots,m\}$, since $|z_j(\omega_j)|$ is a tempered
random variable and the mapping $t\mapsto \ln
|z_j(\theta_t\omega_j)|$ is $\PP$-a.s. continuous, it follows from
Proposition 4.3.3 in \cite{1} that for any $\epsilon_j>0$ there is a
tempered random variable $r_j(\omega_j)>0$ such that
\begin{equation*}
\frac{1}{r_j(\omega_j)}\leq|z_j(\omega_j)|\leq r_j(\omega_j),
\end{equation*}
where $r_j(\omega_j)$ satisfies, for $\PP$-a.s. $\omega\in\Omega$,
\begin{equation}\label{2l}
e^{-\epsilon_j|t|}r_j(\omega_j)\leq r_j(\theta_t\omega_j)\leq
e^{\epsilon_j|t|}r_j(\omega_j),\quad t\in\mathbb{R}.
\end{equation}

Taking $\epsilon_1=\epsilon_2=\cdots=\epsilon_m=\epsilon$, then we
have
\begin{equation*}
\|z(\theta_t\omega)\|\leq\sum\limits_{j=1}^{m}|z_j(\theta_t\omega_j)|\cdot\|h_j\|\leq\sum\limits_{j=1}^{m}r_j(\theta_t\omega_j)\|h_j\|\leq
e^{\epsilon|t|}\sum\limits_{j=1}^{m}r_j(\omega_j)\|h_j\|.
\end{equation*}
Let $r(\omega)=\sum_{j=1}^{m}r_j(\omega_j)\|h_j\|$, \eqref{2j}
is satisfied  and (\ref{2k}) is trivial from \eqref{2l}.
\end{proof}

\begin{corollary}\label{tempered-cor1}
For any $\epsilon>0$, there is a tempered random variable
$r':\Omega\mapsto\mathbb{R}^{+}$ such that
\begin{equation*}
\|A^{\frac{1}{2}}z(\theta_{t}\omega)\|\leq
e^{\epsilon|t|}r'(\omega)\quad\text{for all}\quad
t\in\mathbb{R},\,\,\omega\in\Omega,
\end{equation*}
where
$r'(\omega)=\sum_{j=1}^{m}r_j(\omega_j)\|A^{\frac{1}{2}}h_j\|$ satisfies
\begin{equation*}
e^{-\epsilon|t|}r'(\omega)\leq r'(\theta_t\omega)\leq
e^{\epsilon|t|}r'(\omega),\quad t\in\mathbb{R},\,\,\omega\in\Omega.
\end{equation*}
\end{corollary}

\begin{corollary}\label{tempered-cor2}
For any $\epsilon>0$, there is a tempered random variable
$r'':\Omega\mapsto\mathbb{R}^{+}$ such that
\begin{equation*}
\|Az(\theta_{t}\omega)\|\leq
e^{\epsilon|t|}r''(\omega)\quad\text{for all}\quad
t\in\mathbb{R},\,\,\omega\in\Omega
\end{equation*}
where $r''(\omega)=\sum_{j=1}^{m}r_j(\omega_j)\|Ah_j\|$  satisfies
\begin{equation*}
e^{-\epsilon|t|}r''(\omega)\leq r''(\theta_t\omega)\leq
e^{\epsilon|t|}r''(\omega),\quad t\in\mathbb{R},\,\,\omega\in\Omega.
\end{equation*}
\end{corollary}

\section{Existence of Random Attractor}

\qquad In this section, we study the existence of a random
attractor. Throughout this section we assume that
$p_0=2\pi\eta_0=(2\pi,0)^{\top}\in E_{1}$ and $\delta\in (0,1]$ is
such that $a>0$, where $a$ is as in  \eqref{a-eq}. We remark in the
end of this section that such $\delta$ always exists.

The space $D(C)$ can be endowed with the graph norm,
\begin{equation*}
\|Y\|_{\tilde{E}}=\|Y\|_E+\|CY\|_E\quad\text{for}\,\,Y\in D(C).
\end{equation*}
Since $C$ is a closed operator, $D(C)$ is a Banach space under the
graph norm. We denote $(D(C),\|\cdot\|_{\tilde{E}})$ by $\tilde{E}$
and let $\tilde{E}_1=\tilde{E}\cap E_1$, $\tilde{E}_2=\tilde{E}\cap
E_2$.

By Lemma \ref{periodicity-lm} and the fact that operator $C$ has a
zero eigenvalue, we will define a random dynamical system
$\mathbf{Y}$ defined on torus induced from $Y$. Then by properties
of $Y$ restricted on $E_2$, we can prove the existence of a random
attractor of $\mathbf{Y}$. Thus, we can say that $Y$ has a unbounded
random attractor. Now, we define $\mathbf{Y}$.

Let $\mathbb{T}^{1}=E_{1}/{p_0\mathbb{Z}}$ and
$\mathbf{E}=\mathbb{T}^{1}\times E_{2}$. For $Y_0\in E$, let
$\mathbf{Y_0}:=Y_0\,\,(mod\,p_0)=Y_0+p_0\mathbb{Z}\subset E$ denotes
the equivalence class of $Y_0$, which is an element of $\mathbf{E}$.
And the norm on $\mathbf{E}$ is denoted by
\begin{equation*}
\|\mathbf{Y_0}\|_{\mathbf{E}}=\inf\limits_{y\in
p_0\mathbb{Z}}\|Y_0+y\|_{E}.
\end{equation*}
Note that, by Lemma \ref{periodicity-lm},
$Y(t,\omega,Y_0+kp_0)=Y(t,\omega,Y_0)+kp_0,\,\,\forall
k\in\mathbb{Z}$ for $t\geq0$, $\omega\in\Omega$ and $Y_0\in E$. With
this, we define
$\mathbf{Y}:\mathbb{R}^{+}\times\Omega\times\mathbf{E}\rightarrow\mathbf{E}$
by setting
\begin{equation}\label{3a}
\mathbf{Y}(t,\omega,\mathbf{Y_0})=Y(t,\omega,Y_0)\,\,(mod\,p_0),
\end{equation}
where $\mathbf{Y_0}=Y_0\,\,(mod\,p_0)$. It is easy to see that
$\mathbf{Y}:\mathbb{R}^{+}\times\Omega\times\mathbf{E}\rightarrow\mathbf{E}$
is a random dynamical system.

Similarly, the random dynamical system $\phi$ defined in \eqref{2e}
also induces a random dynamical system $\mathbf{\Phi}$ on
$\mathbf{E}$. By \eqref{2e} and \eqref{3a}, $\mathbf{\Phi}$ is
defined by
\begin{equation}\label{3e}
\mathbf{\Phi}(t,\omega,\mathbf{\Phi_0})=\mathbf{Y}(t,\omega,\mathbf{Y_0})+\tilde{z}(\theta_t\omega)\,\,(mod\,p_0),
\end{equation}
where $\mathbf{\Phi_0}=\phi_0\,\,(mod\,p_0)$,
$\tilde{z}(\theta_t\omega)=(0,z(\theta_t\omega))^{\top}$ and
$\mathbf{Y_0}=\mathbf{\Phi_0}-\tilde{z}(\omega)\,\,(mod\,p_0)$.

The main result of this section can now be stated as follows.

\begin{theorem}\label{existence-attractor-thm}
The random dynamical system $\mathbf{Y}$ defined in \eqref{3a} has a
unique random attractor $\omega\mapsto \mathbf{A_0}(\omega)$, where
\begin{equation*}
\mathbf{A_0}(\omega)=\bigcap_{t>0}\overline{\bigcup_{\tau\geq
t}\mathbf{Y}(\tau,\theta_{-\tau}\omega,\mathbf{B_1}(\theta_{-\tau}\omega))},\quad\omega\in\Omega,
\end{equation*}
in which $\omega\mapsto\mathbf{B_1}(\omega)$ is a tempered random
compact attracting set for $\mathbf{Y}$.
\end{theorem}

\begin{corollary}\label{existence-attractor-cor}
The induced random dynamical system $\mathbf{\Phi}$
defined in \eqref{3e} has a random attractor $\omega\mapsto
\mathbf{A}(\omega)$, where
$\mathbf{A}(\omega)=\mathbf{A_0}(\omega)+\tilde{z}(\omega)\,\,(mod\,p_0)$
for all $\omega\in\Omega$.
\end{corollary}
\begin{proof}
It follows from \eqref{3e} and Theorem
\ref{existence-attractor-thm}.
\end{proof}

To prove Theorem \ref{existence-attractor-thm}, we first introduce
the concept of random pseudo-balls and prove a lemma on the
existence of a pseudo-tempered random absorbing pseudo-ball.

\begin{definition}\label{pseudo-ball-def}
Let $R:\Omega\rightarrow\mathbb{R}^{+}$ be a random variable. A
random pseudo-ball $\omega\in\Omega \mapsto B(\omega)\subset E$ with
random radius $\omega\mapsto R(\omega)$ is a set of the form
\begin{equation*}
\omega\mapsto B(\omega)=\{b(\omega)\in E:\|Qb(\omega)\|_E\leq
R(\omega)\}.
\end{equation*}
Furthermore, a random set $\omega\mapsto B(\omega)\subset  E$ is called
pseudo-tempered provided $\omega\mapsto QB(\omega)$ is a tempered
random set in $E$, i.e., for $\PP$-a.s.$\omega\in\Omega$,
\begin{equation*}
\lim\limits_{t\rightarrow\infty}e^{-\beta t}\sup\{\|Qb\|_{E}:b\in
B(\theta_{-t}\omega)\}=0\quad\text{for all}\quad\beta>0.
\end{equation*}
\end{definition}
Notice that any random pseudo-ball $\omega\mapsto B(\omega)$ in $E$
has the form $\omega\mapsto E_1\times QB(\omega)$, where
$\omega\mapsto QB(\omega)$ is a random ball in $E_2$, which implies
the measurability of $\omega\mapsto B(\omega)$.

By Definition \ref{pseudo-ball-def}, if $\omega\mapsto B(\omega)$ is a random
pseudo-ball in $E$, then $\omega\mapsto B(\omega)\,\,(mod\,p_0)$ is
random bounded set in $\mathbf{E}$. And if $\omega\mapsto B(\omega)$
is a pseudo-tempered random set in $E$, then $\omega\mapsto
B(\omega)\,\,(mod\,p_0)$ is tempered random set in $\mathbf{E}$.

\begin{lemma}\label{existence-attracting-pseudo-ball-lm}
Let $a>0$. Then there exists a tempered random set $\omega\mapsto
\mathbf{B_0}(\omega):=B_0(\omega)\,\,(mod\,p_0)$ in $\mathbf{E}$
such that, for any tempered random set $\omega\mapsto
\mathbf{B}(\omega):=B(\omega)\,\,(mod\,p_0)$ in $\mathbf{E}$, there
is a $T_{\mathbf{B}}(\omega)>0$ such that
\begin{equation*}
\mathbf{Y}(t,\theta_{-t}\omega,\mathbf{B}(\theta_{-t}\omega))\subset
\mathbf{B_0}(\omega)\quad \text{for all}\quad t\geq
T_{\mathbf{B}}(\omega),\,\,\omega\in\Omega,
\end{equation*}
where $\omega\mapsto B_0(\omega)$ is a random pseudo-ball in $E$
with random radius $\omega\mapsto R_0(\omega)$ and $\omega\mapsto
B(\omega)$ is any pseudo-tempered random set in $E$.
\end{lemma}
\begin{proof}
For $\omega\in\Omega$, we obtain from \eqref{2d} that
\begin{equation}\label{3b}
Y(t,\omega,Y_0(\omega))=e^{Ct}Y_0(\omega)+\int_{0}^{t}e^{C(t-s)}F(\theta_s\omega,Y(s,\omega,Y_0(\omega)))ds.
\end{equation}
The projection of \eqref{3b} to $E_2$ is
\begin{equation}\label{3c}
QY(t,\omega,Y_0(\omega))=e^{Ct}QY_0(\omega)+\int_{0}^{t}e^{C(t-s)}QF(\theta_s\omega,Y(s,\omega,Y_0(\omega)))ds.
\end{equation}
By replacing $\omega$ by $\theta_{-t}\omega$, it follows from
\eqref{3c} that
\begin{equation*}
QY(t,\theta_{-t}\omega,Y_0(\theta_{-t}\omega))=e^{Ct}QY_0(\theta_{-t}\omega)+\int_{0}^{t}e^{C(t-s)}QF(\theta_{s-t}\omega,Y(s,\theta_{-t}\omega,Y_0(\theta_{-t}\omega)))ds,
\end{equation*}
and
it then follows from Lemma \ref{projection-lm} and $Q^2=Q$ that
\begin{equation}
\begin{split}
&\|QY(t,\theta_{-t}\omega,Y_0(\theta_{-t}\omega))\|_{E}\\
&\quad\quad\leq
e^{-at}\|QY_0(\theta_{-t}\omega)\|_{E}+\int_{0}^{t}e^{-a(t-s)}\|F(\theta_{s-t}\omega,Y(s,\theta_{-t}\omega,Y_0(\theta_{-t}\omega)))\|_{E}ds.
\end{split}\label{3d}
\end{equation}
By \eqref{2i}, Lemma \ref{tempered-lm1} and Corollary
\ref{tempered-cor1} with $\epsilon=\frac{a}{2}$,
\begin{equation*}
\begin{split}
&\|F(\theta_{s-t}\omega,Y(s,\theta_{-t}\omega,Y_0(\theta_{-t}\omega)))\|_{E}\\
&\quad\quad=\Big(\frac{\alpha^2}{4}\|z(\theta_{s-t}\omega)\|^2+\|(1-\frac{\alpha}{2})z(\theta_{s-t}\omega)-\sin(Y_u)+f\|^2+\|A^{\frac{1}{2}}z(\theta_{s-t}\omega)\|^2\\
&\quad\quad\quad-\delta\lambda_1\|z(\theta_{s-t}\omega)-\overline{z(\theta_{s-t}\omega)}\|^2\Big)^{\frac{1}{2}}\\
&\quad\quad\leq\Big((\alpha^2-3\alpha+3)\|z(\theta_{s-t}\omega)\|^2+3\|\sin(Y_u)\|^2+3\|f\|^2+\|A^{\frac{1}{2}}z(\theta_{s-t}\omega)\|^2\Big)^{\frac{1}{2}}\\
&\quad\quad\leq\Big((\alpha^2-3\alpha+3)e^{a(t-s)}(r(\omega))^2+e^{a(t-s)}(r'(\omega))^2+3|U|+3\|f\|^2\Big)^{\frac{1}{2}}\\
&\quad\quad\leq
a_1e^{\frac{a}{2}(t-s)}r(\omega)+e^{\frac{a}{2}(t-s)}r'(\omega)+a_2,
\end{split}
\end{equation*}
where $Y_u$ satisfies
$Y(s,\theta_{-t}\omega,Y_0(\theta_{-t}\omega))=(Y_u,Y_v)^{\top}$,
$a_1=\sqrt{\alpha^2-3\alpha+3}$, $a_2=\sqrt{3|U|+3\|f\|^2}$ and
$|U|$ is the Lebesgue measure of $U$. We find from \eqref{3d} that
\begin{equation*}
\|QY(t,\theta_{-t}\omega,Y_0(\theta_{-t}\omega))\|_{E}\leq
e^{-at}\|QY_0(\theta_{-t}\omega)\|_{E}+\frac{2}{a}(1-e^{-\frac{a}{2}t})(a_1r(\omega)+r'(\omega))+\frac{a_2}{a}(1-e^{-at}).
\end{equation*}
Now for $\omega\in\Omega$, define
\begin{equation*}
R_0(\omega)=\frac{4}{a}(a_1r(\omega)+r'(\omega))+\frac{2a_2}{a}.
\end{equation*}
Then, for any pseudo-tempered random set $\omega\mapsto B(\omega)$
in $E$ and any $Y_0(\theta_{-t}\omega)\in B(\theta_{-t}\omega)$,
there is a $T_{B}(\omega)>0$ such that for $t\geq T_{B}(\omega)$,
\begin{equation*}
\|QY(t,\theta_{-t}\omega,Y_0(\theta_{-t}\omega))\|_{E}\leq
R_0(\omega),\,\,\omega\in\Omega,
\end{equation*}
which implies
\begin{equation*}
Y(t,\theta_{-t}\omega,B(\theta_{-t}\omega))\subset B_0(\omega)\quad
\text{for all}\quad t\geq T_{B}(\omega),\,\,\omega\in\Omega,
\end{equation*}
where $\omega\mapsto B_0(\omega)$ is the random pseudo-ball centered
at origin with random radius $\omega\mapsto R_0(\omega)$. In fact,
$\omega\mapsto R_0(\omega)$ is a tempered random variable since
$\omega\mapsto r(\omega)$ and $\omega\mapsto r'(\omega)$ are
tempered random variables. Then the measurability of random
pseudo-tempered ball $\omega\mapsto B_0(\omega)$ is obtained from
Definition \ref{pseudo-ball-def} and $\omega\mapsto B_0(\omega)$ is
a random pseudo-ball. Hence, $\omega\mapsto
\mathbf{B_0}(\omega):=B_0(\omega)\,\,(mod\,p_0)$ is a tempered
random ball in $\mathbf{E}$. It then follows from the definition of
$\mathbf{Y}$ that
\begin{equation*}
\mathbf{Y}(t,\theta_{-t}\omega,\mathbf{B}(\theta_{-t}\omega))\subset
\mathbf{B_0}(\omega)\quad \text{for all}\quad t\geq
T_{\mathbf{B}}(\omega),\,\,\omega\in\Omega,
\end{equation*}
where $T_{\mathbf{B}}(\omega)=T_{B}(\omega)$ for $\omega\in\Omega$.
This complete the proof.
\end{proof}

We now prove Theorem \ref{existence-attractor-thm}.

\begin{proof}[Proof of Theorem \ref{existence-attractor-thm}]
By Theorem \ref{existence-attractor-thm-rds}, it suffices to prove
the existence of a random attracting set which restricted on $E_2$
is tempered and compact, i.e., there exists a random set
$\omega\mapsto B_1(\omega)$ such that $\omega\mapsto QB_1(\omega)$
is tempered and compact in $E_2$ and for any pseudo-tempered random
set $\omega\mapsto B(\omega)$ in $E$,
\begin{equation*}
d_{H}(Y(t,\theta_{-t}\omega,B(\theta_{-t}\omega)),B_1(\omega))\rightarrow0\quad\text{as}\quad
t\rightarrow\infty,\,\,\omega\in\Omega,
\end{equation*}
where $d_{H}$ is the Hausdorff semi-distance. Since pseudo-tempered
random sets in $E$ are absorbed by the random absorbing set
$\omega\mapsto B_0(\omega)$, it suffices to prove that
\begin{equation}\label{3f}
d_{H}(Y(t,\theta_{-t}\omega,B_0(\theta_{-t}\omega)),B_1(\omega))\rightarrow0\quad\text{as}\quad
t\rightarrow\infty,\,\,\omega\in\Omega.
\end{equation}
Clearly, if such a $\omega\mapsto B_1(\omega)$ exists, then
$\omega\mapsto\mathbf{B_1}(\omega):={B_1}(\omega)\,\,(mod\,p_0)$ is
a tempered random compact attracting set for $\mathbf{Y}$. We next
show that \eqref{3f} holds.

By the superposition principle, \eqref{2c} with initial data
$Y_0(\omega)$ can be decomposed into
\begin{equation}\label{3g}
\dot{Y}_1=CY_1+F(\theta_t\omega,Y(t,\omega,Y_0(\omega))),\quad
Y_{10}(\omega)=0
\end{equation}
and
\begin{equation}\label{3h}
\dot{Y}_2=CY_2,\quad Y_{20}(\omega)=Y_0(\omega),
\end{equation}
where $Y(t,\omega,Y_0(\omega))$ is the solution of \eqref{2c} with
initial data $Y_0(\omega)\in B_0(\omega)$. Let
$Y_1(t,\omega,Y_{10}(\omega))$ and $Y_2(t,Y_{20}(\omega))$ be
solutions of \eqref{3g} and \eqref{3h}, respectively. We now give
some estimations of $Y_1(t,\omega,Y_{10}(\omega))$ and
$Y_2(t,Y_{20}(\omega))$, which ensure the existence of a random
attracting set which restricted on $E_2$ is tempered and compact.

We first estimate $Y_2(t,Y_{20}(\omega))$. Clearly, \eqref{3h} is a
linear problem. It is easy to see that
\begin{equation*}
Y_2(t,Y_{20}(\omega))=e^{Ct}Y_{20}(\omega),
\end{equation*}
which implies (with $\omega$ being replaced by $\theta_{-t}\omega$)
that
\begin{equation}\label{aux-eqq1}
\|QY_2(t,Y_{20}(\theta_{-t}\omega))\|_E\leq\|e^{Ct}Q\|\cdot\|QY_{20}(\theta_{-t}\omega)\|_E\leq e^{-a t}R_0(\theta_{-t}\omega)\rightarrow 0\quad\text{as}\quad
t\rightarrow\infty.
\end{equation}

For $Y_1(t,\omega,Y_{10}(\omega))$, we show that it is bounded by a
tempered random bounded closed set in $\tilde{E}$, which then is
compact in $E$ since $\tilde{E}$ is compactly imbedded in $E$. Note
that
\begin{equation}\label{3i}
Y_1(t,\omega,Y_{10}(\omega))=\int_{0}^{t}e^{C(t-s)}F(\theta_s\omega,Y(s,\omega,Y_0(\omega)))ds,
\end{equation}
it then follows that
\begin{equation}\label{3j}
\|QY_1(t,\theta_{-t}\omega,Y_{10}(\theta_{-t}\omega))\|_E\leq
\frac{2}{a}(1-e^{-\frac{a}{2}t})(a_1r(\omega)+r'(\omega))+\frac{a_2}{a}(1-e^{-at}),
\end{equation}
where $a_1=\sqrt{\alpha^2-3\alpha+3}$ and
$a_2=\sqrt{3|U|+3\|f\|^2}$ are the same as in the proof of Lemma \ref{existence-attracting-pseudo-ball-lm}, $|U|$ denotes the Lebbesgue measure of $U$.

We next estimate
$CQY_1(t,\theta_{-t}\omega,Y_{10}(\theta_{-t}\omega))$. We find from
\eqref{3i} that
\begin{equation*}
\begin{split}
CQY_1(t,\theta_{-t}\omega,Y_{10}(\theta_{-t}\omega))
&=\int_{0}^{t}e^{C(t-s)}CQF(\theta_{s-t}\omega,Y(s,\theta_{-t}\omega,Y_0(\theta_{-t}\omega)))ds\\
&=\int_{0}^{t}e^{C(t-s)}CF(\theta_{s-t}\omega,Y(s,\theta_{-t}\omega,Y_0(\theta_{-t}\omega)))ds.
\end{split}
\end{equation*}
Then,
\begin{equation}\label{3k}
\|CQY_1(t,\theta_{-t}\omega,Y_{10}(\theta_{-t}\omega))\|_E\leq
\int_{0}^{t}e^{-a(t-s)}\|CF(\theta_{s-t}\omega,Y(s,\theta_{-t}\omega,Y_0(\theta_{-t}\omega)))\|_Eds.
\end{equation}
Obviously,
\begin{equation*}
CF(\theta_{s-t}\omega,Y(s,\theta_{-t}\omega,Y_0(\theta_{-t}\omega)))=\begin{pmatrix}-\sin(Y_u)+f+(1-\alpha)z(\theta_{s-t}\omega)\\\alpha\sin(Y_u)-\alpha
f-\alpha(1-\alpha)-Az(\theta_{s-t}\omega)\end{pmatrix},
\end{equation*}
where $Y_u$ satisfies
$Y(s,\theta_{-t}\omega,Y_0(\theta_{-t}\omega))=(Y_u,Y_v)^{\top}$. By
\eqref{2i}, Lemma \ref{tempered-lm1}, Corollary \ref{tempered-cor1}
and Corollary \ref{tempered-cor2} with $\epsilon=\frac{a}{2}$,
\begin{equation*}
\begin{split}
&\|CF(\theta_{s-t}\omega,Y(s,\theta_{-t}\omega,Y_0(\theta_{-t}\omega)))\|_E^2\\
&\quad\quad\leq
\frac{7}{4}\alpha^2\|\sin(Y_u)\|^2+\frac{7}{4}\alpha^2\|f\|^2+\frac{7}{4}\alpha^2(1-\alpha)^2\|z(\theta_{s-t}\omega)\|^2+4\|Az(\theta_{s-t}\omega)\|^2\\
&\quad\quad\quad+3\|A^{\frac{1}{2}}\sin(Y_u)\|^2+3\|A^{\frac{1}{2}}f\|^2+3(1-\alpha)^2\|A^{\frac{1}{2}}z(\theta_{s-t}\omega)\|^2\\
&\quad\quad\leq
a_3^2+\frac{7}{4}\alpha^2(1-\alpha)^2e^{a(t-s)}(r(\omega))^2+3(1-\alpha)^2e^{a(t-s)}(r'(\omega))^2\\
&\quad\quad\quad+4e^{a(t-s)}(r''(\omega))^2+3\|A^{\frac{1}{2}}\sin(Y_u)\|^2\\
&\quad\quad\leq
\Big(a_3+\frac{\sqrt{7}}{2}\alpha|1-\alpha|e^{\frac{a}{2}(t-s)}r(\omega)+\sqrt{3}|1-\alpha|e^{\frac{a}{2}(t-s)}r'(\omega)\\
&\quad\quad\quad+2e^{\frac{a}{2}(t-s)}r''(\omega)+\sqrt{3}\|A^{\frac{1}{2}}\sin(Y_u)\|\Big)^2,
\end{split}
\end{equation*}
where
$a_3=\sqrt{\frac{7}{4}\alpha^2|U|+\frac{7}{4}\alpha^2\|f\|^2+3\|A^{\frac{1}{2}}f\|^2}$.
Then, \eqref{3k} implies
\begin{equation}
\begin{split}
&\|CQY_1(t,\theta_{-t}\omega,Y_{10}(\theta_{-t}\omega))\|_E\\
&\quad\quad\leq\frac{a_3}{a}(1-e^{-at})+\sqrt{3}\int_{0}^{t}e^{-a(t-s)}\|A^{\frac{1}{2}}\sin(Y_u)\|ds\\
&\quad\quad\quad+\frac{2}{a}\Big(\frac{\sqrt{7}}{2}\alpha|1-\alpha|r(\omega)+\sqrt{3}|1-\alpha|r'(\omega)+2r''(\omega)\Big)(1-e^{-\frac{a}{2}t}).
\end{split}\label{3l}
\end{equation}
For the integral on the right-hand side of (\ref{3l}), we note that
\begin{equation*}
\|A^{\frac{1}{2}}\sin(Y_u)\|\leq\|A^{\frac{1}{2}}Y_u\|\leq
a_4\|QY(s,\theta_{-t}\omega,Y_0(\theta_{-t}\omega))\|_E,
\end{equation*}
where $a_4=\sqrt{2/(2-\delta)}$. Since
\begin{equation*}
\begin{split}
&\|QY(s,\theta_{-t}\omega,Y_0(\theta_{-t}\omega))\|_{E}\\
&\quad\quad\leq
e^{-as}\|QY_0(\theta_{-t}\omega)\|_{E}+\int_{0}^{s}e^{-a(s-\tau)}\|F(\theta_{\tau-t}\omega,Y(\tau,\theta_{-t}\omega,Y_0(\theta_{-t}\omega)))\|_{E}d\tau,
\end{split}
\end{equation*}
we find that
\begin{equation}
\begin{split}
&\int_{0}^{t}e^{-a(t-s)}\|A^{\frac{1}{2}}\sin(Y_u)\|ds\\
&\quad\quad\leq a_4te^{-at}\|QY_0(\theta_{-t}\omega)\|_{E}
+a_4\int_{0}^{t}\int_{0}^{s}e^{-a(t-\tau)}\|F(\theta_{\tau-t}\omega,Y(\tau,\theta_{-t}\omega,Y_0(\theta_{-t}\omega)))\|_{E}d\tau
ds\\
&\quad\quad\leq a_4te^{-at}\|QY_0(\theta_{-t}\omega)\|_{E}\\
&\quad\quad\quad+a_4\int_{0}^{t}\Bigg(\frac{2}{a}\Big(a_1r(\omega)+r'(\omega)\Big)(e^{-\frac{a}{2}(t-s)}-e^{-\frac{a}{2}t})+\frac{a_2}{a}(e^{-a(t-s)}-e^{-at})\Bigg)ds\\
&\quad\quad=a_4te^{-at}\|QY_0(\theta_{-t}\omega)\|_{E}+\frac{2a_4}{a}\Big(a_1r(\omega)+r'(\omega)\Big)\Big(\frac{2}{a}(1-e^{-\frac{a}{2}t})-te^{-\frac{a}{2}t}\Big)\\
&\quad\quad\quad+\frac{a_2a_4}{a}\Big(\frac{1}{a}(1-e^{-at})-te^{-at}\Big).
\end{split}\label{3m}
\end{equation}
Combining \eqref{3j}, \eqref{3l} and \eqref{3m}, there is a
$T(\omega)>0$ such that for all $t\geq T(\omega)$,
\begin{align}
\label{aux-eqq2}
&\|QY_1(t,\theta_{-t}\omega,Y_{10}(\theta_{-t}\omega))\|_{\tilde{E}}\nonumber\\
&\quad\quad=\|QY_1(t,\theta_{-t}\omega,Y_{10}(\theta_{-t}\omega))\|_E+\|CQY_1(t,\theta_{-t}\omega,Y_{10}(\theta_{-t}\omega))\|_E\nonumber\\
&\quad\quad\leq R_1(\omega),
\end{align}

where
$R_1(\omega)=a_5r(\omega)+a_6r'(\omega)+\frac{8}{a}r''(\omega)+a_7$
 is a  tempered random variable, in which
$a_5=\frac{4a_1+2\sqrt{7}\alpha|1-\alpha|}{a}+\frac{8\sqrt{3}a_1a_4}{a^2}$,
$a_6=\frac{4+4\sqrt{3}|1-\alpha|}{a}+\frac{8\sqrt{3}a_4}{a^2}$ and
$a_7=\frac{2a_2+2a_3}{a}+\frac{2\sqrt{3}a_2a_4}{a^2}$.

Now, let $\omega\mapsto B_1(\omega)$ be the random pseudo-ball in
$\tilde{E}$ centered at origin with random radius $\omega\mapsto
R_1(\omega)$, then $\omega\mapsto B_1(\omega)$ is tempered and
measurable. By \eqref{aux-eqq1}, \eqref{aux-eqq2} and
\begin{equation*}
QY(t,\theta_{-t}\omega,\phi_0(\theta_{-t}\omega))=QY_1(t,\theta_{-t}\omega,Y_{10}(\theta_{-t}\omega))+QY_2(t,Y_{20}(\theta_{-t}\omega)),
\end{equation*}
we have for $\omega\in\Omega$,
\begin{equation*}
d_{H}(Y(t,\theta_{-t}\omega,B_0(\theta_{-t}\omega)),B_1(\omega))\rightarrow0\quad\text{as}\quad
t\rightarrow\infty.
\end{equation*}
Then by the compact embedding of $\tilde{E}$ into $E$,
$\omega\mapsto QB_1(\omega)$ is compact in $E_2$, which implies that
$\omega\mapsto\mathbf{B_1}(\omega):={B_1}(\omega)\,\,(mod\,p_0)$ is
a tempered random compact attracting set for $\mathbf{Y}$. Thus by
Theorem \ref{existence-attractor-thm-rds}, $\mathbf{Y}$ has a unique  random
attractor $\omega\mapsto \mathbf{A_0}(\omega)$, where
\begin{equation*}
\mathbf{A_0}(\omega)=\bigcap_{t>0}\overline{\bigcup_{\tau\geq
t}\mathbf{Y}(\tau,\theta_{-\tau}\omega,\mathbf{B_1}(\theta_{-\tau}\omega))},\quad\omega\in\Omega.
\end{equation*}
This completes the proof.
\end{proof}

\begin{remark}
\label{existence-attractor-rk}
\begin{itemize}
\item[(1)] For any $\alpha>0$ and
$\lambda_1=K\tilde{\lambda}_1>0$ (see \eqref{lambdas-eq}), there is
a $\delta\in (0,1]$ such that $a>0$ holds, where $a$ is as in
\eqref{a-eq} and $\tilde{\lambda}_1$ is the smallest positive
eigenvalue of $-\triangle$ and a constant.

\item[(2)] We can say that the random dynamical $Y$(or $\phi$) has a unique
random attractor in the sense that the induced random dynamical
system $\mathbf{Y}$(or $\mathbf{\Phi}$) has a unique random
attractor, and we will say that $Y$(or $\phi$) has a unique random
attractor directly in the sequel. We denote the random attractor of
$Y$ and $\phi$ by $\omega\mapsto A_0(\omega)$ and $\omega\mapsto
A(\omega)$ respectively. Indeed, $\omega\mapsto A_0(\omega)$ and
$\omega\mapsto A(\omega)$ satisfy
\begin{equation*}
\mathbf{A_0}(\omega)=A_0(\omega)\,\,(mod\,p_0),\quad
\mathbf{A}(\omega)=A(\omega)\,\,(mod\,p_0),\quad\omega\in\Omega.
\end{equation*}

\item[(3)] For the deterministic damped sine-Gordon equation with homogeneous
Neumann boundary condition, the authors proved in \cite{19} that the
random attractor is a horizontal curve provided that $\alpha$ and
$K$ are sufficiently large. Similarly, we expect that the random
attractor $\omega\mapsto A(\omega)$ of $\phi$ has the similar
property, i.e., $A(\omega)$ is a horizontal curve for each
$\omega\in\Omega$ provided that $\alpha$ and $K$ are sufficiently
large. We prove that this is true in next section.

\item[(4)] By (2),  system \eqref{main-eq}-\eqref{main-bc} is dissipative
(i.e. it possesses a random attractor). In  section 6, we will show
that \eqref{main-eq}-\eqref{main-bc} with sufficiently large
$\alpha$ and $K$ also has a rotation number and hence all the
solutions tend to oscillate with the same frequency eventually.
\end{itemize}
\end{remark}

\section{One-dimensional Random Attractor}

\qquad In this section, we apply the theory established in \cite{4}
to show that the random attractor of $Y$ (or $\phi$) is
one-dimensional provided that $\alpha$ and $K$ are sufficiently
large. This method has been used by Chow, Shen and Zhou \cite{3} to
systems of coupled noisy oscillators. Throughout this section we
assume that $p_0=2\pi\eta_0=(2\pi,0)^{\top}\in E_{1}$ and $a>4L_{F}$
(see \eqref{a-eq}  for the definition of  $a$ and see
\eqref{lipschitz-constant-f} for the upper bound of $L_F$). We
remark in the end of this section that this condition can be
satisfied provided that $\alpha$ and $K$ are sufficiently large.

\begin{definition}\label{horizontal-curve-def}
Suppose $\{\Phi^{\omega}\}_{\omega\in\Omega}$ is a family of maps
from $E_1$ to $E_2$ and $n\in\NN$. A family of graphs $\omega\mapsto
\ell(\omega)\equiv\{(p,\Phi^{\omega}(p)):p\in E_{1}\}$ is said to be
a random $np_0$-periodic horizontal curve if $\omega\mapsto
\ell(\omega)$ is a random set and
$\{\Phi^{\omega}\}_{\omega\in\Omega}$ satisfy the Lipshcitz
condition
\begin{equation*}
\|\Phi^{\omega}(p_1)-\Phi^{\omega}(p_2)\|_{E}\leq
\|p_1-p_2\|_{E}\quad\text{for all}\quad p_1,p_2\in
E_1,\,\,\omega\in\Omega
\end{equation*}
and the periodic condition
\begin{equation*}
\Phi^{\omega}(p+np_0)=\Phi^{\omega}(p)\quad\text{for all}\quad p\in
E_1,\,\,\omega\in\Omega.
\end{equation*}
\end{definition}

Clearly, for any $\omega\in\Omega$, $\ell(\omega)$ is a
deterministic $np_0$-periodic horizontal curve. When $n=1$, we
simply call it a horizontal curve.

\begin{lemma}\label{horizontal-curve-lm}
Let $a>4L_{F}$. Suppose that $\omega\mapsto \ell(\omega)$ is a
random $np_0$-periodic horizontal curve in $E$. Then, $\omega\mapsto
Y(t,\omega,\ell(\omega))$ is also a random $np_0$-periodic
horizontal curve in $E$ for all $t>0$. Moreover, $\omega\mapsto
Y(t,\theta_{-t}\omega,\ell(\theta_{-t}\omega))$ is a random
$np_0$-periodic horizontal curve for all $t>0$.
\end{lemma}
\begin{proof}
First, since $Y$ is a random dynamical system and $\omega\mapsto
\ell(\omega)$ is a random set in $E$, $\omega\mapsto
Y(t,\omega,\ell(\omega))$ and $\omega\mapsto
Y(t,\theta_{-t}\omega,\ell(\theta_{-t}\omega))$ are random sets in
$E$ for all $t>0$. We next show the Lipschitz condition and periodic
condition.

It is sufficient to prove the Lipschitz condition and periodic
condition valid for $\omega\mapsto \ell(\omega)$ in $D(C)$ since
$D(C)$ is dense in $E$. Clearly, for $\omega\in\Omega$ and $t>0$,
\begin{equation*}
Y(t,\omega,\ell(\omega))=\{(PY(t,\omega,p+\Phi^{\omega}(p)),QY(t,\omega,p+\Phi^{\omega}(p))):p\in
E_{1}\cap D(C)\}.
\end{equation*}
For $p_1,p_2\in E_1\cap D(C)$, $p_1\neq p_2$, let
$Y_i(t,\omega)=Y(t,\omega,p_i+\Phi^\omega(p_i))$, $i=1,2$,
$p(t,\omega)=P(Y_1(t,\omega)-Y_2(t,\omega))$ and
$q(t,\omega)=Q(Y_1(t,\omega)-Y_2(t,\omega))$, where $P$, $Q$ are
defined as in section 1. We have by Lemma \ref{projection-lm}
\begin{equation*}
\begin{split}
PY_i(t,\omega)&=e^{Ct}P(p_i+\Phi^\omega(p_i))+\int_{0}^{t}e^{C(t-s)}PF(\theta_s\omega,Y_i(s,\omega))ds\\
&=P(p_i+\Phi^\omega(p_i))+\int_{0}^{t}PF(\theta_s\omega,Y_i(s,\omega))ds,
\quad i=1,2,
\end{split}
\end{equation*}
and then,
$\frac{d}{dt}PY_i(t,\omega)=PF(\theta_t\omega,Y_i(t,\omega)),\,\,
i=1,2$, it then follows that
\begin{equation}
\begin{split}
\frac{d}{dt}p(t,\omega)&=\frac{d}{dt}P(Y_1(t,\omega)-Y_2(t,\omega))\\
&=P(F(\theta_t\omega,Y_1(t,\omega))-F(\theta_t\omega,Y_2(t,\omega))).
\end{split}\label{4a}
\end{equation}
Since $p(t,\omega)+q(t,\omega)=Y_1(t,\omega)-Y_2(t,\omega)$,
\begin{equation*}
\begin{split}
\frac{d}{dt}(p(t,\omega)+q(t,\omega))
&=\frac{d}{dt}(Y_1(t,\omega)-Y_2(t,\omega))\\
&=C(Y_1(t,\omega)-Y_2(t,\omega))+F(\theta_t\omega,Y_1(t,\omega))-F(\theta_t\omega,Y_2(t,\omega)),
\end{split}
\end{equation*}
then, by the orthogonal decomposition,
\begin{equation}
\begin{split}
\frac{d}{dt}q(t,\omega)&=C(Y_1(t,\omega)-Y_2(t,\omega))+Q(F(\theta_t\omega,Y_1(t,\omega))-F(\theta_t\omega,Y_2(t,\omega)))\\
&=Cq(t,\omega)+Q(F(\theta_t\omega,Y_1(t,\omega))-F(\theta_t\omega,Y_2(t,\omega))).
\end{split}\label{4b}
\end{equation}

We find from \eqref{4a} that
\begin{equation*}
\begin{split}
\frac{d}{dt}\|p(t,\omega)\|_{E}^{2}&=2\big\langle
p(t,\omega),\frac{d}{dt}p(t,\omega)\big\rangle_{E}\\
&\geq-2\|p(t,\omega)\|_{E}\cdot\|P(F(\theta_t\omega,Y_1(t,\omega))-F(\theta_t\omega,Y_2(t,\omega)))\|_{E}\\
&\geq-2L_{F}(\|p(t,\omega)\|_{E}^2+\|p(t,\omega)\|_{E}\|q(t,\omega)\|_{E}).
\end{split}
\end{equation*}
Similarly, by \eqref{4b} and Lemma \ref{projection-lm},
\begin{equation*}
\frac{d}{dt}\|q(t,\omega)\|_{E}^{2}\leq-2a\|q(t,\omega)\|_{E}^{2}+2L_{F}(\|p(t,\omega)\|_{E}\|q(t,\omega)\|_{E}+\|q(t,\omega)\|_{E}^2).
\end{equation*}
Because $a>4L_{F}$, if there is a $t_0\geq0$ such that
$\|q(t_0,\omega)\|_{E}=\|p(t_0,\omega)\|_{E}$ and since
$\|p(t,\omega)\|_{E}\neq0$ for $t\geq0$, then
\begin{equation*}
\frac{d}{dt}\Big|_{t=t_0}\Big(\|q(t,\omega)\|_{E}^{2}-\|p(t,\omega)\|_{E}^{2}\Big)\leq(8L_{F}-2a)\|q(t_0,\omega)\|_{E}^{2}<0,
\end{equation*}
which means that there is a $\bar{t}_0>t_0$ such that for
$t\in(t_0,\bar{t}_0)$,
\begin{equation*}
\begin{split}
\|q(t,\omega)\|_{E}^2-\|p(t,\omega)\|_{E}^2&<\|q(0,\omega)\|_{E}^2-\|p(0,\omega)\|_{E}^2\\
&=\|\Phi^{\omega}(p_1)-\Phi^{\omega}(p_2)\|_{E}^2-\|p_1-p_2\|_{E}^2\\
&\leq0,
\end{split}
\end{equation*}
namely, $\|q(t,\omega)\|_{E}<\|p(t,\omega)\|_{E}$ for
$t\in(t_0,\bar{t}_0)$.

If there is another $t_1\geq \bar{t}_0$ such that
$\|q(t_1,\omega)\|_{E}=\|p(t_1,\omega)\|_{E}$, then
\begin{equation*}
\frac{d}{dt}\Big|_{t=t_1}\Big(\|q(t,\omega)\|_{E}^{2}-\|p(t,\omega)\|_{E}^{2}\Big)\leq(8L_{F}-2a)\|q(t_1,\omega)\|_{E}^{2}<0,
\end{equation*}
which means that there is a $\bar{t}_1>t_1$ such that for
$t\in(t_1,\bar{t}_1)$, $\|q(t,\omega)\|_{E}<\|p(t,\omega)\|_{E}$.
Continue this process, we have for all $t\geq0$,
$\|q(t,\omega)\|_{E}\leq\|p(t,\omega)\|_{E}$, i.e.,
\begin{equation*}
\|Q(Y_1(t,\omega)-Y_2(t,\omega))\|_{E}\leq\|P(Y_1(t,\omega)-Y_2(t,\omega))\|_{E},
\end{equation*}
which shows that $\omega\mapsto Y(t,\omega,\ell(\omega))$ satisfies
the Lipschitz condition in Definition \ref{horizontal-curve-def}.

We next show the periodic condition. We find from Lemma
\ref{periodicity-lm} that
\begin{equation*}
Y(t,\omega,p+\Phi^{\omega}(p))+np_0=Y(t,\omega,p+np_0+\Phi^{\omega}(p))..
\end{equation*}
Since $\Phi^{\omega}(p)=\Phi^{\omega}(p+np_0)$,
$Y(t,\omega,p+\Phi^{\omega}(p))+np_0=Y(t,\omega,p+np_0+\Phi^{\omega}(p+np_0))$.
It follows that
\begin{equation*}
QY(t,\omega,p+\Phi^{\omega}(p))=QY(t,\omega,p+np_0+\Phi^{\omega}(p+np_0)).
\end{equation*}
Consequently, $\omega\mapsto Y(t,\omega,\ell(\omega))$ is a random
$np_0$-periodic horizontal curve for all $t>0$.

Moreover, for any fixed $\omega\in\Omega$ and $t>0$,
$\bar{\omega}=\theta_{-t}\omega\in\Omega$ is fixed. Then,
$Y(t,\bar{\omega},\ell(\bar{\omega}))$ is a deterministic
$np_0$-periodic horizontal curve, which yields the assertion.
\end{proof}

Choose $\gamma\in(0,\frac{a}{2})$ such that
\begin{equation}\label{4c}
\frac{2}{\alpha}\Bigg(\frac{1}{\gamma}+\frac{1}{a-2\gamma}\Bigg)<1,
\end{equation}
where $\frac{2}{\alpha}$ is the upper bound of the Lipschitz
constant of $F$ (see \eqref{lipschitz-constant-f}). We remark in the
end of this section that such a $\gamma$ exists provided that
$\alpha$ and $K$ are sufficiently large. We next show the main
result in this section.

\begin{theorem}\label{one-dimension-thm}
Assume that $a>4L_{F}$ and that there is a
$\gamma\in(0,\frac{a}{2})$ such that \eqref{4c} holds. Then the
random attractor $\omega\mapsto A_0(\omega)$ of the random dynamical
system $Y$ is a random horizontal curve.
\end{theorem}
\begin{proof}
By the equivalent relation between $\phi$ and $Y$, we mainly focus
on equation \eqref{2c}, which can be viewed as a deterministic
system with a random parameter $\omega\in\Omega$. We write it here
as \eqref{2c}$_{\omega}$ for some fixed $\omega\in\Omega$.

Observe that the linear part of \eqref{2c}$_{\omega}$, i.e.
\begin{equation}\label{aux-eqq3}
\dot Y=CY
\end{equation}
has a one-dimensional center space $E^c={\rm span}\{(1,0)\}=E_1$ and
a one co-dimensional stable space $E^s=E_2$. We first show that \eqref{2c}$_{\omega}$
has a one-dimensional invariant manifold, denoted by $W(\omega)$,  and will show later that $W(\omega)$
exponentially attracts all the solutions of \eqref{2c}$_{\omega}$.

Let $F^{\omega}(t,Y)=F(\theta_{t}\omega, Y)$, $\omega\in\Omega$. For
fixed $\omega\in\Omega$, consider the following integral equation
\begin{equation}\label{4d}
\tilde{Y}(t)=e^{Ct}\xi+\int_{0}^{t}e^{C(t-s)}PF^{\omega}(s,\tilde{Y}(s))ds+\int_{-\infty}^{t}e^{C(t-s)}QF^{\omega}(s,\tilde{Y}(s))ds,\quad
t\leq0,
\end{equation}
where $\xi=P\tilde{Y}(0)\in E_1$. For $g:(-\infty,0]\rightarrow E$
such that $\sup_{t\leq0}\|e^{\gamma t}g(t)\|_{E}<\infty$, define
\begin{equation*}
(Lg)(t)=\int_{0}^{t}e^{C(t-s)}Pg(s)ds+\int_{-\infty}^{t}e^{C(t-s)}Qg(s)ds,
\quad t\leq0.
\end{equation*}
It is easy to see that
\begin{equation*}
\sup_{t\leq0}\|e^{\gamma
t}(Lg)(t)\|_{E}\leq\Bigg(\frac{1}{\gamma}+\frac{1}{a-\gamma}\Bigg)\sup_{t\leq0}\|e^{\gamma
t}g(t)\|_E\leq\Bigg(\frac{1}{\gamma}+\frac{1}{a-2\gamma}\Bigg)\sup_{t\leq0}\|e^{\gamma
t}g(t)\|_E,
\end{equation*}
which means that $\|L\|\leq\frac{1}{\gamma}+\frac{1}{a-2\gamma}$.
Then, Theorem 3.3 in \cite{4} shows that for any $\xi\in E_1$,
equation \eqref{4d} has a unique solution
$\tilde{Y}^{\omega}(t,\xi)$ satisfying $\sup_{t\leq0}\|e^{\gamma
t}\tilde{Y}^{\omega}(t,\xi)\|_{E}<\infty$. Let
\begin{equation*}
h(\omega,\xi)=Q\tilde{Y}^{\omega}(0,\xi)=\int_{-\infty}^{0}e^{-Cs}QF^{\omega}(s,\tilde{Y}^{\omega}(s,\xi))ds,\quad
\omega\in\Omega.
\end{equation*}
Let
\begin{equation*}
\begin{split}
W(\omega)=\{\xi+h(\omega,\xi):\xi\in E_1\},\quad\omega\in\Omega.
\end{split}
\end{equation*}
For any $\epsilon\in(0,\gamma)$ in Lemma \ref{tempered-lm1} and
Corollary \ref{tempered-cor1}, we have
\begin{equation}\label{4e}
\|h(\theta_{-t}\omega,\xi)\|_E\leq\frac{1}{a-\epsilon}(a_1r(\omega)+r'(\omega))e^{\epsilon
t}+\frac{a_2}{a},\quad t\geq0.
\end{equation}

Observe that
\begin{equation*}
\begin{split}
\tilde{Y}^{\omega}(t,\xi)&=e^{Ct}\xi+\int_{0}^{t}e^{C(t-s)}PF^{\omega}(s,\tilde{Y}^{\omega}(s,\xi))ds+\int_{-\infty}^{t}e^{C(t-s)}QF^{\omega}(s,\tilde{Y}(s,\omega,\xi))ds\\
&=e^{Ct}(\xi+h(\omega,\xi))+\int_{0}^{t}e^{C(t-s)}F^{\omega}(s,\tilde{Y}^{\omega}(s,\xi))ds,
\end{split}
\end{equation*}
i.e., $\tilde{Y}^{\omega}(t,\xi)$ is the solution of \eqref{2c} with
initial data $\xi+h(\omega,\xi)$ for $t\leq0$. Thus, for
$Y_0(\omega)=\xi+h(\omega,\xi)\in W(\omega)$, there is a
negative continuation of $Y(t,\omega,Y_0(\omega))$, i.e.,
\begin{equation}\label{4f}
Y(t,\omega,Y_0(\omega))=\tilde{Y}^{\omega}(t,\xi),\quad t\leq0.
\end{equation}
Moreover, for $t\geq0$, we obtain from \eqref{2d} and \eqref{4f}
that
\begin{equation*}
\begin{split}
&Y(t,\omega,Y_0(\omega))\\
&\quad\quad=e^{Ct}(\xi+h(\omega,\xi))+\int_{0}^{t}e^{C(t-s)}F(\theta_s\omega,Y(s,\omega,Y_0(\omega)))ds\\
&\quad\quad=e^{Ct}\xi+\int_{0}^{t}e^{C(t-s)}F(\theta_s\omega,Y(s,\omega,Y_0(\omega)))ds+\int_{-\infty}^{0}e^{C(t-s)}QF^{\omega}(s,\tilde{Y}^{\omega}(s,\xi))ds\\
&\quad\quad=e^{Ct}\xi+\int_{0}^{t}e^{C(t-s)}F(\theta_s\omega,Y(s,\omega,Y_0(\omega)))ds+\int_{-\infty}^{0}e^{C(t-s)}QF(\theta_s\omega,Y(s,\omega,Y_0(\omega)))ds\\
&\quad\quad=e^{Ct}\xi+\int_{0}^{t}e^{C(t-s)}PF(\theta_s\omega,Y(s,\omega,Y_0(\omega)))ds+\int_{-\infty}^{t}e^{C(t-s)}QF(\theta_s\omega,Y(s,\omega,Y_0(\omega)))ds\\
&\quad\quad=e^{Ct}\Big(\xi+\int_{0}^{t}e^{-Cs}PF(\theta_s\omega,Y(s,\omega,Y_0(\omega)))ds\Big)+\int_{-\infty}^{0}e^{-Cs}QF(\theta_{t+s}\omega,Y(t+s,\omega,Y_0(\omega)))ds.
\end{split}
\end{equation*}
Then by the uniqueness of solution of \eqref{4d} for fixed
$\omega\in\Omega$, we have
\begin{equation*}
\begin{split}
&h\Big(\theta_{t}\omega,e^{Ct}\Big(\xi+\int_{0}^{t}e^{-Cs}PF(\theta_s\omega,Y(s,\omega,Y_0(\omega)))ds\Big)\Big)\\
&\quad\quad=\int_{-\infty}^{0}e^{-Cs}QF(\theta_{t+s}\omega,Y(t+s,\omega,Y_0(\omega)))ds,
\end{split}
\end{equation*}
and then for $t\geq0$,
\begin{equation}\label{4g}
Y(t,\omega,W(\omega))=W(\theta_t\omega).
\end{equation}
By \eqref{4f} and \eqref{4g}, $W(\omega)$ is an invariant manifold
of \eqref{2c}$_{\omega}$.

Next, we show that $W(\omega)$ attracts the solutions of
\eqref{2c}$_{\omega}$, more precisely, for the given
$\omega\in\Omega$, we prove the existence of a stable foliation
$\{W_s(\omega,Y_0):Y_0\in W(\omega)\}$ of the invariant manifold
$W(\omega)$ of \eqref{2c}$_{\omega}$. Consider the following
integral equation
\begin{equation}
\begin{split}
\hat{Y}(t)&=e^{Ct}\eta+\int_{0}^{t}e^{C(t-s)}Q\Big(F^{\omega}(s,\hat{Y}(s)+Y^{\omega}(s,\xi+h(\omega,\xi)))\\
&\quad\quad\quad\quad\quad\quad\quad\quad\quad-F^{\omega}(s,Y^{\omega}(s,\xi+h(\omega,\xi)))\Big)ds\\
&\quad+\int_{\infty}^{t}e^{C(t-s)}P\Big(F^{\omega}(s,\hat{Y}(s)+Y^{\omega}(s,\xi+h(\omega,\xi)))\\
&\quad\quad\quad\quad\quad\quad\quad-F^{\omega}(s,Y^{\omega}(s,\xi+h(\omega,\xi)))\Big)ds,\quad
t\geq0,
\end{split}\label{4h}
\end{equation}
where $\xi+h(\omega,\xi)\in W(\omega)$, $\eta=Q\hat{Y}(0)\in
E_2$ and
$Y^{\omega}(t,\xi+h(\omega,\xi)):=Y(t,\omega,\xi+h(\omega,\xi))$,
$t\geq0$ is the solution of \eqref{2c} with initial data
$\xi+h(\omega,\xi)$ for fixed $\omega\in\Omega$. Theorem 3.4 in
\cite{4} shows that for any $\xi\in E_1$ and $\eta\in E_2$, equation
\eqref{4h} has a unique solution $\hat{Y}^{\omega}(t,\xi,\eta)$
satisfying $\sup_{t\geq0}\|e^{\gamma
t}\hat{Y}^{\omega}(t,\xi,\eta)\|_{E}<\infty$ and for any $\xi\in
E_1$, $\eta_1,\,\eta_2\in E_2$,
\begin{equation}\label{4i}
\sup_{t\geq0}e^{\gamma
t}\|\hat{Y}^{\omega}(t,\xi,\eta_1)-\hat{Y}^{\omega}(t,\xi,\eta_2)\|_{E}\leq
M\|\eta_1-\eta_2\|_{E}.
\end{equation}
where
$M=\frac{1}{1-\frac{2}{\alpha}\big(\frac{1}{\gamma}+\frac{1}{a-2\gamma}\big)}$.
Let
\begin{equation*}
\begin{split}
\hat{h}(\omega,\xi,\eta)&=\xi+P\hat{Y}^{\omega}(0,\xi,\eta)\\
&=\xi+\int_{\infty}^{0}e^{-Cs}P\Big(F^{\omega}(s,\hat{Y}^{\omega}(s,\xi,\eta)+Y^{\omega}(s,\xi+h(\omega,\xi)))\\
&\quad\quad\quad\quad\quad\quad\quad-F^{\omega}(s,Y^{\omega}(s,\xi+h(\omega,\xi)))\Big)ds.
\end{split}
\end{equation*}
Then,
$W_{s}(\omega,\xi+h(\omega,\xi))=\{\eta+h(\omega,\xi)+\hat{h}(\omega,\xi,\eta):\eta\in
E_2\}$ is the stable foliation of $W(\omega)$ at
$\xi+h(\omega,\xi)$.

Observe that
\begin{equation}
\begin{split}
&\hat{Y}^{\omega}(t,\xi,\eta)+Y^{\omega}(t,\xi+h(\omega,\xi))-Y^{\omega}(t,\xi+h(\omega,\xi))\\
&\quad\quad=\hat{Y}^{\omega}(t,\xi,\eta)\\
&\quad\quad=e^{Ct}(\eta+h(\omega,\xi)+\hat{h}(\omega,\xi,\eta)-\xi-h(\omega,\xi))\\
&\quad\quad\quad+\int_{0}^{t}e^{C(t-s)}\Big(F^{\omega}(s,\hat{Y}^{\omega}(s,\xi,\eta)+Y^{\omega}(s,\xi+h(\omega,\xi)))\\
&\quad\quad\quad\quad\quad\quad\quad\quad\quad-F^{\omega}(s,Y^{\omega}(s,\xi+h(\omega,\xi)))\Big)ds
\end{split}\label{4j}
\end{equation}
and
\begin{equation}
\begin{split}
&Y^{\omega}(t,\eta+h(\omega,\xi)+\hat{h}(\omega,\xi,\eta))-Y^{\omega}(t,\xi+h(\omega,\xi))\\
&\quad\quad=e^{Ct}(\eta+h(\omega,\xi)+\hat{h}(\omega,\xi,\eta)-\xi-h(\omega,\xi))\\
&\quad\quad\quad+\int_{0}^{t}e^{C(t-s)}\Big(F^{\omega}(s,Y^{\omega}(s,\eta+h(\omega,\xi)+\hat{h}(\omega,\xi,\eta)))\\
&\quad\quad\quad\quad\quad\quad\quad\quad\quad-F^{\omega}(s,Y^{\omega}(s,\xi+h(\omega,\xi)))\Big)ds.
\end{split}\label{4k}
\end{equation}
Comparing \eqref{4j} with \eqref{4k}, we find that
\begin{equation}\label{4l}
\hat{Y}^{\omega}(t,\xi,\eta)=Y^{\omega}(t,\eta+h(\omega,\xi)+\hat{h}(\omega,\xi,\eta))-Y^{\omega}(t,\xi+h(\omega,\xi)),\quad
t\geq0.
\end{equation}
In addition, if $\eta=0$, then by the uniqueness of solution of
\eqref{4h}, $\hat{Y}^{\omega}(t,\xi,0)\equiv0$ for $t\geq0$, which
associates with \eqref{4i} and \eqref{4l} show that
\begin{equation}\label{4m}
\sup_{t\geq0}e^{\gamma
t}\|Y^{\omega}(t,\eta+h(\omega,\xi)+\hat{h}(\omega,\xi,\eta))-Y^{\omega}(t,\xi+h(\omega,\xi))\|_{E}\leq
M\|\eta\|_E
\end{equation}
for any $\xi\in E_1$ and $\eta\in E_2$.

We now claim that $\omega\mapsto W(\omega)$ is the random
attractor of $Y$. Let $\omega\mapsto B(\omega)$ be any
pseudo-tempered random set in E. For any $\omega\mapsto
Y_0(\omega)\in\omega\mapsto B(\omega)$, there is
$\omega\mapsto\xi(\omega)\in E_1$ such that
\begin{equation*}
\begin{split}
Y_0(\theta_{-t}\omega)\in
W_s(\theta_{-t}\omega,\xi(\theta_{-t}\omega)+h(\theta_{-t}\omega,\xi(\theta_{-t}\omega))).
\end{split}
\end{equation*}
Let
$\eta(\theta_{-t}\omega)=QY_0(\theta_{-t}\omega)-h(\theta_{-t}\omega,\xi(\theta_{-t}\omega))$.
By \eqref{4e}, it is easy to see that
\begin{equation*}
\begin{split}
\sup_{Y_0(\theta_{-t}\omega)\in
B(\theta_{-t}\omega)}\|\eta(\theta_{-t}\omega)\|\leq\sup_{Y_0(\theta_{-t}\omega)\in
B(\theta_{-t}\omega)}\|QY_0(\theta_{-t}\omega)\|+\frac{1}{a-\epsilon}(a_1r(\omega)+r'(\omega))e^{\epsilon
t}+\frac{a_2}{a}.
\end{split}
\end{equation*}
It then follows from \eqref{4m} and the fact that $\omega\mapsto
QB(\omega)$ is tempered that
\begin{equation*}
\begin{split}
&\sup_{Y_0(\theta_{-t}\omega)\in
B(\theta_{-t}\omega)}\|Y(t,\theta_{-t}\omega,Y_0(\theta_{-t}\omega))-Y(t,\theta_{-t}\omega,\xi(\theta_{-t}\omega)+h(\theta_{-t}\omega,\xi(\theta_{-t}\omega)))\|_{E}\\
&\quad\quad\leq Me^{-\gamma t}\sup_{Y_0(\theta_{-t}\omega)\in
B(\theta_{-t}\omega)}\|\eta(\theta_{-t}\omega)\|_E\\
&\quad\quad\leq Me^{-\gamma t}\sup_{Y_0(\theta_{-t}\omega)\in
B(\theta_{-t}\omega)}\|QY_0(\theta_{-t}\omega)\|+\frac{M}{a-\epsilon}(a_1r(\omega)+r'(\omega))e^{(\epsilon-\gamma)t}+\frac{a_2M}{a}e^{-\gamma t}\\
&\quad\quad\rightarrow0\quad\text{as}\quad t\rightarrow\infty,
\end{split}
\end{equation*}
which associates with \eqref{4g} lead to
\begin{equation*}
\begin{split}
d_{H}(Y(t,\theta_{-t}\omega,B(\theta_{-t}\omega)),W(\omega))\rightarrow0\quad\text{as}\quad
t\rightarrow\infty.
\end{split}
\end{equation*}
Therefore, $A_0(\omega)=W(\omega)$ for $\omega\in\Omega$. Next, we
show that $\omega\mapsto A_0(\omega)$ is a random horizontal curve.
In fact, for some random horizontal curve
$\omega\mapsto\ell(\omega)$ in $E$, for example,
$\ell(\omega)\equiv\{(p,\Phi^{\omega}(p)):\Phi^{\omega}(p)=c, p\in
E_{1}\}$, $\omega\in\Omega$, where $c\in E_2$ is constant, it must
be contained in some pseudo-tempered random set, for example
$\omega\mapsto B_{2\|c\|_{E}}(\omega)$, where
$B_{2\|c\|_{E}}(\omega)$ is a pseudo-ball with radius $2\|c\|_{E}$.
Then, for $\omega\in\Omega$,
\begin{equation*}
\begin{split}
d_{H}(Y(t,\theta_{-t}\omega,\ell(\theta_{-t}\omega)),A_0(\omega))\rightarrow0\quad\text{as}\quad
t\rightarrow\infty,
\end{split}
\end{equation*}
which means that
$\lim_{t\rightarrow\infty}Y(t,\theta_{-t}\omega,\ell(\theta_{-t}\omega))\subset
A_0(\omega)$. Since $A_0(\omega)$ is one-dimensional, we have for
$\omega\in\Omega$,
\begin{equation*}
\begin{split}
A_0(\omega)=\lim_{t\rightarrow\infty}Y(t,\theta_{-t}\omega,\ell(\theta_{-t}\omega)).
\end{split}
\end{equation*}
It then follows from Lemma \ref{horizontal-curve-lm}
that $\omega\mapsto A_0(\omega)$ is a random horizontal curve.
\end{proof}

\begin{corollary}
\label{one-dimension-cor} Assume that $a>4L_{F}$ and that there is a
$\gamma\in(0,\frac{a}{2})$ such that \eqref{4c} holds. Then the
random attractor $\omega\mapsto A(\omega)$ of the random dynamical
system $\phi$ is a random horizontal curve.
\end{corollary}
\begin{proof}
It follows from Corollary \ref{existence-attractor-cor},  Remark
\ref{existence-attractor-rk}
 and Theorem \ref{one-dimension-thm}.
\end{proof}

\begin{remark}
\label{one-dimension-rk}
At the beginning of this section, we assume that $a>4L_{F}$. Since
$a=\frac{\alpha}{2}-|\frac{\alpha}{2}-\frac{\delta\lambda_1}{\alpha}|$
and $L_{F}\leq\frac{2}{\alpha}$, we can take $\alpha$, $\lambda_1$
satisfying
$\frac{\alpha}{2}-\Big|\frac{\alpha}{2}-\frac{\delta\lambda_1}{\alpha}\Big|>\frac{8}{\alpha}$,
where $\lambda_1$ is the smallest positive eigenvalue of $A$ and its
value is determined by the diffusion coefficient $K$. On the other
hand, we need some $\gamma\in(0,\frac{a}{2})$ such that \eqref{4c}
holds. Note that
\begin{equation*}
\begin{split}
\min_{\gamma\in(0,\frac{a}{2})}\Bigg(\frac{1}{\gamma}+\frac{1}{a-2\gamma}\Bigg)=\Bigg(\frac{1}{\gamma}+\frac{1}{a-2\gamma}\Bigg)\Bigg|_{\gamma=\frac{(2-\sqrt{2})a}{2}}=\frac{\sqrt{2}}{(3\sqrt{2}-4)a},
\end{split}
\end{equation*}
which implies that there exist $\alpha$, $\lambda_1$ satisfying
\begin{equation}\label{4n}
\frac{\alpha}{2}-\Big|\frac{\alpha}{2}-\frac{\delta\lambda_1}{\alpha}\Big|>\frac{2\sqrt{2}}{(3\sqrt{2}-4)\alpha}>\frac{8}{\alpha}.
\end{equation}
Indeed, let $c=\frac{2\sqrt{2}}{3\sqrt{2}-4}$, then for any
$\alpha>\sqrt{2c}$ and $\lambda_1>c$, there is a $\delta>0$
satisfying
\begin{equation*}
\begin{split}
\frac{c}{\lambda_1}<\delta<\min\Big\{\frac{\alpha^2-c}{\lambda_1},1\Big\}
\end{split}
\end{equation*}
such that \eqref{4n} holds.
\end{remark}

\section{Rotation Number}

\qquad In this section, we study the phenomenon of frequency
locking, i.e., the existence of a rotation number of the stochastic
damped sine-Gordon equation \eqref{main-eq}-\eqref{main-bc}, which
characterizes the speed that the solution of
\eqref{main-eq}-\eqref{main-bc} moves around the one-dimensional
random attractor.

\begin{definition}
\label{rotation-number-def} The stochastic damped sine-Gordon
equation \eqref{main-eq} with boundary condition \eqref{main-bc} is
said to have a rotation number $\rho\in\mathbb{R}$ if, for
$\PP$-a.e. $\omega\in\Omega$ and each $\phi_0=(u_0,u_1)^{\top}\in
E$, the limit
$\lim_{t\rightarrow\infty}\frac{P\phi(t,\omega,\phi_0)}{t}$ exists
and
\begin{equation*}
\lim\limits_{t\rightarrow\infty}\frac{P\phi(t,\omega,\phi_0)}{t}=\rho\eta_0,
\end{equation*}
where $\eta_0=(1,0)^{\top}$ is the basis of $E_{1}$.
\end{definition}

We remark that the rotation number of \eqref{main-eq}-\eqref{main-bc} (if exists) is unique.
In fact, assume that $\rho_1$ and $\rho_2$ are rotation numbers of \eqref{main-eq}-\eqref{main-bc}.
Then there is $\omega\in\Omega$ such that for any $\phi_0\in E$,
$$\rho_1\eta_0=\lim_{t\to\infty}\frac{ P\phi(t,\omega,\phi_0)}{t}=\rho_2\eta_0.
$$
Therefore, $\rho_1=\rho_2$ and then the rotation number of \eqref{main-eq}-\eqref{main-bc} (if exists) is unique.

From \eqref{2e}, we have
\begin{equation}\label{5a}
\frac{P\phi(t,\omega,\phi_0)}{t}=\frac{PY(t,\omega,Y_0(\omega))}{t}+\frac{P(0,z(\theta_t\omega))^{\top}}{t},
\end{equation}
where $\phi_0=(u_0,u_1)^{\top}$ and
$Y_0(\omega)=(u_0,u_1-z(\omega))^{\top}$. By Lemma 2.1 in \cite{8},
it is easy to prove that
$\lim_{t\rightarrow\infty}\frac{P(0,z(\theta_t\omega))^{\top}}{t}=(0,0)^{\top}$.
Thus, it sufficient to prove the existence of the rotation number of
the random system (\ref{2c}).

By the random dynamical system $\mathbf{Y}$ defined in (\ref{3a}),
we define the corresponding skew-product semiflow
$\mathbf{\Theta}_t:\Omega\times\mathbf{E}\rightarrow\Omega\times\mathbf{E}$
for $t\geq0$ by setting
\begin{equation*}
\mathbf{\Theta}_t(\omega,\mathbf{Y_0})=(\theta_t\omega,\mathbf{Y}(t,\omega,\mathbf{Y_0})).
\end{equation*}
Obviously,
$(\Omega\times\mathbf{E},\,\mathcal{F}\times\mathcal{B},\,(\mathbf{\Theta}_t)_{t\geq0})$
is a measurable dynamical system, where
$\mathcal{B}=\mathcal{B}(\mathbf{E})$ is the Borel $\sigma$-algebra
of $\mathbf{E}$.

\begin{lemma}\label{invariant-measure-lm}
There is a measure $\mu$ on
$\Omega\times\mathbf{E}$ such that
$(\Omega\times\mathbf{E},\,\mathcal{F}\times\mathcal{B},\,\mu,\,(\mathbf{\Theta}_t)_{t\geq0})$
becomes an ergodic metric dynamical system.
\end{lemma}
\begin{proof}
Let $Pr_{\Omega}(\mathbf{E})$ be the set of all random probability
measures on $\mathbf{E}$ and $Pr_{\PP}(\Omega\times\mathbf{E})$ be
the set of all probability measures on $\Omega\times\mathbf{E}$ with
marginal $\PP$. We know from Proposition 3.3 and Proposition 3.6 in
\cite{6} that $Pr_{\Omega}(\mathbf{E})$ and
$Pr_{\PP}(\Omega\times\mathbf{E})$ are isomorphism. Moreover, both
$Pr_{\Omega}(\mathbf{E})$ and $Pr_{\PP}(\Omega\times\mathbf{E})$ are
convex, and the convex structure is preserved by this isomorphism.

Let $\Gamma=\{\omega\mapsto\mu_{\omega}\in
Pr_{\Omega}(\mathbf{E}):\PP\text{-a.s.}
\mu_{\omega}(\mathbf{A_0}(\omega))=1,\,\,\omega\mapsto\mu_{\omega}\,\,
\text{is invariant for}\,\,\mathbf{Y}\}$. Clearly, $\Gamma$ is
convex. Since $\omega\mapsto\mathbf{A_0}(\omega)$ is the random
attractor of $\mathbf{Y}$, we obtain from Corollary 6.13 in \cite{6}
that $\Gamma\neq\emptyset$. Let $\omega\mapsto\mu_{\omega}$ be an
extremal point of $\Gamma$. Then, by the isomorphism between
$Pr_{\Omega}(\mathbf{E})$ and $Pr_{\PP}(\Omega\times\mathbf{E})$ and
Lemma 6.19 in \cite{6}, the corresponding measure $\mu$ on
$\Omega\times\mathbf{E}$ of $\omega\mapsto\mu_{\omega}$ is
$(\mathbf{\Theta}_t)_{t\geq0}$-invariant and ergodic. Thus,
$(\Omega\times\mathbf{E},\,\mathcal{F}\times\mathcal{B},\,\mu,\,(\mathbf{\Theta}_t)_{t\geq0})$
is an ergodic metric dynamical system.
\end{proof}

We next show a simple lemma which will be used. For any
$p_i=(s_i,0)^{\top}\in E_1$, $i=1,2$, we define
\begin{equation*}
p_1\leq p_2\quad\text{if}\quad s_1\leq s_2.
\end{equation*}
Then we have
\begin{lemma}\label{orientation-preserving}
Suppose that $a>4L_F$. Let $\ell$ be any deterministic
$np_0$-periodic horizontal curve ($\ell$ satisfies the Lipschitz and
periodic condition in Definition \ref{horizontal-curve-def}). For
any $Y_1,\,\,Y_2\in\ell$ with $PY_1\leq PY_2$, there holds
\begin{equation}\label{orientation inequality}
PY(t,\omega,Y_1)\leq PY(t,\omega,Y_2)\quad\text{for}\,\,
t>0,\,\,\omega\in\Omega.
\end{equation}
\end{lemma}
\begin{proof}
Clearly, if $PY_1=PY_2$, then \eqref{orientation inequality} holds.
We now prove that \eqref{orientation inequality} holds for
$PY_1<PY_2$. If not, then by the continuity of $Y$ with respect to
$t$, there is a $t_0>0$ such that
$PY(t_0,\omega,Y_1)=PY(t_0,\omega,Y_2)$, which implies that
$Y(t_0,\omega,Y_1)=Y(t_0,\omega,Y_2)$ since $Y(t_0,\omega,Y_1)$ and
$Y(t_0,\omega,Y_2)$ belong to the same deterministic $np_0$-periodic
horizontal curve $Y(t_0,\omega,\ell)$, which leads to a
contradiction. The lemma is thus proved.
\end{proof}

We now show the main result in this section.
\begin{theorem}\label{existence-rotation-number-thm}
Assume that $a>4L_{F}$. Then the
rotation number of \eqref{2c} exists.
\end{theorem}
\begin{proof}
Note that
\begin{equation*}
\frac{PY(t,\omega,Y_0)}{t}=\frac{PY_0}{t}+\frac{1}{t}\int_{0}^{t}PF(\theta_s\omega,Y(s,\omega,Y_0))ds.
\end{equation*}
Since
$F(\theta_s\omega,Y(s,\omega,Y_0)+kp_0)=F(\theta_s\omega,Y(s,\omega,Y_0)),\,\,\forall
k\in\mathbb{Z}$, we can identify
$F(\theta_s\omega,\mathbf{Y}(s,\omega,\mathbf{Y_0}))$ with
$F(\theta_s\omega,Y(s,\omega,Y_0))$. Precisely, define
$h:E\rightarrow\mathcal{E}$, $Y\mapsto\{Y\}$, where $\mathcal{E}$ is
the collection of all singleton sets of $E$, i.e.
$\mathcal{E}=\{\{Y\}:Y\in E\}$ (see Remark \ref{interpret-of-E} for
more details of the space $\mathcal{E}$). Clearly, $h$ is a
homeomorphism from $E$ to $\mathcal{E}$. Then,
\begin{equation*}
F(\theta_s\omega,Y(s,\omega,Y_0))=h^{-1}(F(\theta_s\omega,\mathbf{Y}(s,\omega,\mathbf{Y_0}))).
\end{equation*}
Thus,
\begin{equation}
\begin{split}
\frac{PY(t,\omega,Y_0)}{t}&=\frac{PY_0}{t}+\frac{1}{t}\int_{0}^{t}Ph^{-1}(F(\theta_s\omega,\mathbf{Y}(s,\omega,\mathbf{Y_0})))ds\\
&=\frac{PY_0}{t}+\frac{1}{t}\int_{0}^{t}\mathbf{F}(\mathbf{\Theta}_s(\omega,\mathbf{Y_0}))ds.\\
\end{split}\label{5b}
\end{equation}
where $\mathbf{F}=P\circ h^{-1}\circ F\in
L^{1}(\Omega\times\mathbf{E},\,\mathcal{F}\times\mathcal{B},\,\mu)$.
Let $t\rightarrow\infty$ in (\ref{5b}),
$\lim_{t\rightarrow\infty}\frac{PY_0}{t}=(0,0)^{\top}$ and by Lemma
5.2 and Ergodic Theorems in \cite{1}, there exist a constant
$\rho\in\mathbb{R}$ such that
\begin{equation*}
\lim_{t\rightarrow\infty}\frac{1}{t}\int_{0}^{t}\mathbf{F}(\mathbf{\Theta}_s(\omega,\mathbf{Y_0}))ds=\rho\eta_0,
\end{equation*}
which means
\begin{equation*}
\lim_{t\rightarrow\infty}\frac{PY(t,\omega,Y_0)}{t}=\rho\eta_0
\end{equation*}
for $\mu$-a.e.$(\omega,Y_0)\in\Omega\times E$.  Thus, there is
$\Omega^*\subset\Omega$ with $\PP(\Omega^*)=1$ such that for any
$\omega\in\Omega^*$,
 there is $Y_0^{*}(\omega)\in E$ such that
\begin{equation*}
\lim_{t\rightarrow\infty}\frac{PY(t,\omega,Y_0^{*}(\omega))}{t}=\rho\eta_0.
\end{equation*}
By Lemma
\ref{periodicity-lm}, we have that for any $n\in\mathbb{N}$ and $\omega\in\Omega^*$,
\begin{equation}\label{5c}
\lim_{t\rightarrow\infty}\frac{PY(t,\omega,Y^{*}_0(\omega)\pm
np_0)}{t}=\lim_{t\rightarrow\infty}\frac{PY(t,\omega,Y^{*}_0(\omega))\pm
np_0}{t}=\rho\eta_0.
\end{equation}

Now for any $\omega\in\Omega^*$ and any $Y_0\in E$, there is
$n_0(\omega)\in\mathbb{N}$ such that
\begin{equation*}
PY^{*}_0(\omega)-n_0(\omega)p_0\leq PY_0\leq
PY^{*}_0(\omega)+n_0(\omega)p_0
\end{equation*}
and there is a $n_0(\omega)p_0$-periodic horizontal curve $l_0(
\omega)$ such that $Y^*_0(\omega)-n_0(\omega)p_0$, $Y_0$, $Y^*_0(\omega)+n_0(\omega)p_0\in l_0(\omega)$.
Then by Lemma \ref{orientation-preserving}, we have
\begin{equation*}
PY(t,\omega,Y^{*}_0(\omega)-n_0(\omega)p_0)\leq
PY(t,\omega,Y_0)\leq
PY(t,\omega,Y^{*}_0(\omega)+n_0(\omega)p_0),
\end{equation*}
which together with \eqref{5c} implies that for  any
$\omega\in\Omega^*$ and any $Y_0\in E$,
\begin{equation*}
\lim_{t\rightarrow\infty}\frac{PY(t,\omega,Y_0)}{t}=\rho\eta_0.
\end{equation*}
Consequently, for any a.e. $\omega\in\Omega$ and any $Y_0\in E$,
\begin{equation*}
\lim_{t\rightarrow\infty}\frac{PY(t,\omega,Y_0)}{t}=\rho\eta_0.
\end{equation*}
The theorem is thus proved.
\end{proof}

\begin{corollary}\label{existence-rotation-number-cor}
Assume that $a>4L_{F}$. Then the rotation number of the stochastic
damped sine-Gordon equation \eqref{main-eq} with the boundary
condition \eqref{main-bc} exists.
\end{corollary}
\begin{proof}
It follows from \eqref{5a} and Theorem
\ref{existence-rotation-number-thm}.
\end{proof}

\begin{remark}\label{interpret-of-E}
We first note that the space $\mathcal{E}=\{\{Y\}:Y\in E\}$ in the
proof of Theorem \ref{existence-rotation-number-thm} is a linear
space according to the linear structure defined by
\begin{equation*}
\begin{split}
\alpha\{X\}+\beta\{Y\}=\{\alpha X+\beta
Y\},\quad\text{for}\,\,\alpha,\beta\in\mathbb{R},\,\,\{X\},\{Y\}\in\mathcal{E}.
\end{split}
\end{equation*}
Also, for $\{X\},\{Y\}\in\mathcal{E}$, we define
\begin{equation}\label{inner-on-E}
\begin{split}
\langle\{X\},\{Y\}\rangle_{\mathcal{E}}=\langle X,Y\rangle_{E}.
\end{split}
\end{equation}
It is easy to verify that the functional
$\langle\cdot,\cdot\rangle_{\mathcal{E}}:\mathcal{E}\times\mathcal{E}\rightarrow\mathbb{R}$
defined by \eqref{inner-on-E} is bilinear, symmetric and positive,
thus defining the scalar product in $\mathcal{E}$ over $\mathbb{R}$.
Moreover, the completeness of $\mathcal{E}$ is from the completeness
of $E$. Hence, $\mathcal{E}$ is a Hilbert space.
\end{remark}

\begin{remark} In the proof of Theorem \ref{existence-rotation-number-thm}, we used an ergodic invariant measure
$\mu$ of
$(\Omega\times\mathbf{E},\,\mathcal{F}\times\mathcal{B},\,\mu,\,(\mathbf{\Theta}_t)_{t\geq0})$.
It should be pointed out that
 the measure $\mu$ on $\Omega\times\mathbf{E}$ that makes
$(\Omega\times\mathbf{E},\,\mathcal{F}\times\mathcal{B},\,\mu,\,(\mathbf{\Theta}_t)_{t\geq0})$
becomes an ergodic metric dynamical system may not be unique, because
the convex set $\Gamma$ in the proof of Lemma
\ref{invariant-measure-lm} may have more than one extremal points.
However, as mentioned above, the rotation number in Theorem
\ref{existence-rotation-number-thm} and Corollary
\ref{existence-rotation-number-cor} are independent of  $\mu$ and are unique.
\end{remark}

\vskip 0.5cm

\noindent \textbf{Acknowledgement}. We would like to thank the
referee for carefully reading the manuscript and making very useful
suggestions.

\end{document}